\documentclass[a4paper,12pt]{amsart}
\usepackage{amsmath,amsthm,amssymb}
\usepackage[dvipdfmx]{graphicx}
\usepackage[all]{xy}
\usepackage{graphicx}
\usepackage{here}
\usepackage{color}
\usepackage[top=30truemm,bottom=30truemm,left=25truemm,right=25truemm]{geometry}
\usepackage{ytableau}

\newtheorem{theo}{Theorem}[subsection]
\newtheorem{defi}[theo]{Definition}
\newtheorem{lem}[theo]{Lemma}
\newtheorem{rem}[theo]{Remark}
\newtheorem{prop}[theo]{Proposition}
\newtheorem{cor}[theo]{Corollary}

\newtheorem{ex}[theo]{Example}

\newtheorem{Thm}{Theorem}

\newcommand{\nc}{\newcommand}
\nc{\on}{\operatorname}
\nc{\C}{\mathbb{C}}
\nc{\R}{\mathbb{R}}
\nc{\Q}{\mathbb{Q}}
\nc{\Z}{\mathbb{Z}}
\nc{\N}{\mathbb{N}}

\nc{\bbH}{\mathbb{H}}
\nc{\bbK}{\mathbb{K}}

\nc{\bfa}{\mathbf{a}}
\nc{\bfA}{\mathbf{A}}
\nc{\bfb}{\mathbf{b}}
\nc{\bfi}{\mathbf{i}}
\nc{\bfK}{\mathbf{K}}
\nc{\bfk}{\mathbf{k}}
\nc{\bfx}{\mathbf{x}}
\nc{\bfy}{\mathbf{y}}
\nc{\bfone}{\boldsymbol 1}
\nc{\bfeta}{\boldsymbol \eta}
\nc{\bfkappa}{\boldsymbol \kappa}
\nc{\bfrho}{\boldsymbol \rho}
\nc{\bfsigma}{\boldsymbol \sigma}
\nc{\bfvarsigma}{\boldsymbol \varsigma}
\nc{\bfzeta}{\boldsymbol \zeta}

\nc{\A}{\mathbf{A}}
\nc{\U}{\mathbf{U}}

\nc{\clB}{\mathcal{B}}
\nc{\clC}{\mathcal{C}}
\nc{\clF}{\mathcal{F}}
\nc{\clL}{\mathcal{L}}
\nc{\clM}{\mathcal{M}}
\nc{\clO}{\mathcal{O}}
\nc{\clU}{\mathcal{U}}
\nc{\clUi}{\clU^{\imath}}
\nc{\clX}{\mathcal{X}}
\nc{\clXs}{\clX_{\mathrm{s}}}
\nc{\clY}{\mathcal{Y}}
\nc{\clYs}{\clY_{\mathrm{s}}}
\nc{\clW}{\mathcal{W}}
\nc{\clZ}{\mathcal{Z}}

\nc{\fra}{\mathfrak{a}}
\nc{\frh}{\mathfrak{h}}
\nc{\g}{\mathfrak{g}}
\nc{\frgl}{\mathfrak{gl}}
\nc{\frk}{\mathfrak{k}}
\nc{\fram}{\mathfrak{m}}
\nc{\frn}{\mathfrak{n}}
\nc{\frp}{\mathfrak{p}}
\nc{\frs}{\mathfrak{s}}
\nc{\frt}{\mathfrak{t}}
\nc{\frsl}{\mathfrak{sl}}
\nc{\frso}{\mathfrak{so}}
\nc{\frsp}{\mathfrak{sp}}
\nc{\frsu}{\mathfrak{su}}
\nc{\fru}{\mathfrak{u}}
\nc{\frz}{\mathfrak{z}}

\nc{\fin}{\mathrm{fin}}

\nc{\inv}{^{-1}}
\nc{\qu}{\quad}

\nc{\qqu}{\qquad}
\nc{\la}{\langle}
\nc{\ra}{\rangle}

\nc{\Ker}{\on{Ker}}
\nc{\im}{\on{Im}}
\nc{\Hom}{\on{Hom}}
\nc{\End}{\on{End}}
\nc{\Span}{\on{Span}}
\nc{\id}{\on{id}}
\nc{\Aut}{\on{Aut}}
\nc{\ad}{\on{ad}}
\nc{\ch}{\on{ch}}
\nc{\sgn}{\on{sgn}}
\nc{\tot}{\on{tot}}
\nc{\rk}{\on{rk}}
\nc{\rank}{\on{rank}}
\nc{\Wt}{\on{Wt}}
\nc{\diag}{\on{diag}}
\nc{\Mat}{\on{Mat}}
\nc{\tr}{\on{tr}}
\nc{\Diag}{\on{Diag}}
\nc{\GL}{\on{GL}}
\nc{\SO}{\on{SO}}
\nc{\Sp}{\on{Sp}}
\nc{\gr}{\on{gr}}
\nc{\Ind}{\on{Ind}}
\nc{\Res}{\on{Res}}
\nc{\wt}{\on{wt}}
\nc{\lt}{\on{lt}}
\nc{\lc}{\on{lc}}
\nc{\Br}{\on{Br}}
\nc{\norm}{\on{norm}}
\nc{\amp}{\on{amp}}
\nc{\Int}{\on{int}}
\nc{\Inv}{\on{inv}}
\nc{\pr}{\on{pr}}
\nc{\sh}{\on{sh}}
\nc{\std}{P^{\AI}}

\nc{\AI}{\mathrm{A\!I}}
\nc{\CR}{\mathrm{CR}}
\nc{\ev}{\mathrm{ev}}
\nc{\odd}{\mathrm{odd}}
\nc{\OT}{\mathrm{OT}}
\nc{\Par}{\mathrm{Par}}
\nc{\RS}{\mathrm{RS}}
\nc{\RT}{\mathrm{RT}}
\nc{\Sing}{\mathrm{Sing}}
\nc{\SST}{\mathrm{SST}}
\nc{\STab}{\mathrm{ST}}

\nc{\ol}{\overline}
\nc{\ul}{\underline}
\nc{\hf}{\frac{1}{2}}
\nc{\vphi}{\varphi}
\nc{\vrho}{\varrho}
\nc{\vpi}{\varpi}
\nc{\vep}{\varepsilon}
\nc{\eps}{\epsilon}
\nc{\lm}{\lambda}
\nc{\til}{\widetilde}

\nc{\IF}{\text{ if }}
\nc{\AND}{\text{ and }}
\nc{\OR}{\text{ or }}
\nc{\OW}{\text{ otherwise}}
\nc{\lowerterms}{\text{(lower terms)}}
\nc{\higherterms}{\text{(higher terms)}}
\nc{\lex}{\text{lex}}
\nc{\sesi}{\text{ss}}
\nc{\ST}{\text{ such that }}
\nc{\Forsome}{\text{ for some }}
\nc{\Forall}{\text{ for all }}

\nc{\Atil}{\widetilde{A}}
\nc{\Btil}{\widetilde{B}}
\nc{\Etil}{\widetilde{E}}
\nc{\Ftil}{\widetilde{F}}
\nc{\Itil}{\widetilde{I}}
\nc{\Xtil}{\widetilde{X}}
\nc{\Ytil}{\widetilde{Y}}

\nc{\Ui}{\U^{\imath}}
\nc{\Uidot}{\dot{\U}^\imath}
\nc{\UidotA}{\dot{\U}^\imath_{\bfA}}
\nc{\taui}{\tau^{\imath}}
\nc{\psii}{\psi^{\imath}}
\nc{\wti}{\wt^\imath}

\nc{\Udot}{\dot{\U}}
\nc{\UdotA}{\Udot_{\bfA}}

\nc{\plim}[1][]{\mathop{\varprojlim}\limits_{#1}}
\nc{\ilim}[1][]{\mathop{\varinjlim}\limits_{#1}}

\nc{\TBA}{{\large {\bf \textcolor{red}{To Appear}}}}
\nc{\alert}{\textcolor{red}}

\nc{\Sbox}[1]{\ytableausetup{smalltableaux}
\begin{ytableau}
#1
\end{ytableau}}
\nc{\Cbox}[1]{\ytableausetup{boxsize=normal,centertableaux}
\begin{ytableau}
#1
\end{ytableau}}

\title[A new tableau model for representations of $SO_n$]{A new tableau model for representations of the special orthogonal group}
\author[H. Watanabe]{Hideya Watanabe}
\address{(H. Watanabe) Osaka Central Advanced Mathematical Institute, Osaka Metropolitan University, Osaka, 558-8585, Japan}
\email{watanabehideya@gmail.com}
\date{\today}

\begin{document}
\maketitle

\begin{abstract}
We provide a new tableau model from which one can easily deduce the characters of finite-dimensional irreducible polynomial representations of the special orthogonal group $SO_n(\mathbb{C})$. This model originates from the representation theory of the $\imath$quantum group (also known as the quantum symmetric pair coideal subalgebra) of type $\AI$, and is equipped with a combinatorial structure, which we call $\AI$-crystal structure. This structure enables us to describe combinatorially the tensor product of an $SO_n(\mathbb{C})$-module and a $GL_n(\mathbb{C})$-module, and the branching from $GL_n(\mathbb{C})$ to $SO_n(\mathbb{C})$.
\end{abstract}

\section{Introduction}
Combinatorial objects such as partitions, Young tableaux, and their variants have been effectively used to understand the representation theory of the symmetric group, the general linear group $GL_n = GL_n(\mathbb{C})$, the general linear algebra $\frgl_n = \frgl_n(\C)$, and their variants. For example, the set of standard (resp., semistandard) Young tableaux of a fixed shape parametrizes a basis of the corresponding irreducible representation of the symmetric group (resp., $GL_n$ and $\frgl_n$).

For a better understanding of a representation theory, it is quite useful to construct a concrete combinatorial model which uniformly models a certain class of representations. For example, in the representation theory of reductive groups and their Lie algebras, one may want a combinatorial model from which one can easily deduce the character of a representation under consideration.

The semistandard Young tableau model for $GL_n$ and $\mathfrak{gl}_n$ is such a typical example. Let us see this in more detail.
Let $\Par_n$ denote the set of partitions of length not greater than $n$. For each $\lambda \in \Par_n$, let $\SST_n(\lm)$ denote the set of semistandard Young tableaux of shape $\lm$ with entries in $\{ 1,\ldots,n \}$. To each semistandard Young tableau $T \in \SST_n(\lm)$, a weight $\wt(T) \in \Z^n$ is assigned. Then, $\SST_n(\lm)$ models the finite-dimensional irreducible $\mathfrak{gl}_n$-module (equivalently, $GL_n$-module) $V^{\mathfrak{gl}_n}(\lm)$ of highest weight $\lm$ in the sense that the character of $V^{\mathfrak{gl}_n}(\lm)$ equals the generating function
$$
\sum_{T \in \SST_n(\lm)} \bfx^{\wt(T)} \in \Z[x_1^{\pm 1}, \ldots, x_n^{\pm 1}]
$$
of weights of $\SST_n(\lm)$, where
$$
\bfx^{(a_1,\ldots,a_n)} := x_1^{a_1} \cdots x_n^{a_n}.
$$

The aim of this paper is to introduce a new tableau model for the finite-dimensional irreducible representations of integer highest weight for the special orthogonal algebra $\mathfrak{so}_n = \mathfrak{so}_n(\mathbb{C})$, equivalently, the finite-dimensional irreducible polynomial representations for the special orthogonal group. The isomorphism classes of such representations are parametrized by the set
\begin{align*}
  \begin{split}
    &X_{\mathfrak{so}_n, \text{int}}^+ \\
    &\quad := \begin{cases}
      \{ \nu = (\nu_1,\nu_3,\dots,\nu_{2m-1}) \in \mathbb{Z}^m \mid \nu_1 \geq \nu_3 \geq \dots \geq \nu_{2m-1} \} & \text{ if } n \text{ is odd}, \\
      \{ \nu = (\nu_1,\nu_3,\dots,\nu_{2m-1}) \in \mathbb{Z}^m \mid \nu_1 \geq \nu_3 \geq \dots \geq \nu_{2m-3} \geq |\nu_{2m-1}|  \} & \text{ if } n \text{ is even}
    \end{cases}
  \end{split}
\end{align*}
of dominant integer weights, where
$$
m := \begin{cases}
  \frac{n-1}{2} \qu & \IF n \text{ is odd}, \\
  \frac{n}{2} \qu & \IF n \text{ is even}
\end{cases}
$$
denotes the rank of $\mathfrak{so}_n$.
It is convenient to introduce the following family of $\mathfrak{so}_n$-modules parametrized by $\Par_m$.
For each $\rho \in \Par_m$, set
$$
V^{\mathfrak{so}_n}(\rho) := \begin{cases}
  V^{\mathfrak{so}_n}(\nu_\rho) & \text{ if } \ell(\rho) < \frac{n}{2}, \\
  V^{\mathfrak{so}_n}(\nu_\rho^+) \oplus V^{\mathfrak{so}_n}(\nu_\rho^-) & \text{ if } \ell(\rho) = \frac{n}{2},
\end{cases}
$$
where
$$
\nu_\rho := (\rho_1,\rho_2,\dots,\rho_{\ell(\rho)},0,\dots,0), \ \nu_\rho^\pm := (\rho_1,\rho_2,\dots,\rho_{m-1}, \pm \rho_{m}) \in X_{\mathfrak{so}_n, \text{int}}^+.
$$

Let $\rho \in \Par_m$ and $T \in \SST_n(\rho)$. We say that $T$ is an $\AI$-tableau (the origin of the name will be clear later) if
$$
t_{i,1}^c \leq t_{i,2} \ \Forall 1 \leq i \leq d_2,
$$
where $t_{i,j}$ denotes the $(i,j)$-th entry of $T$, and
$$
d_j := \max \{ k \mid \rho_k \geq j \}, \qu \{ t_{1,1}^c,\ldots,t_{n-d_1,1}^c \} := \{ 1,\ldots,n \} \setminus \{ t_{1,1},\ldots,t_{d_1,1} \}.
$$
For each $\rho \in \Par_m$, set
$$
\SST_n^{\AI}(\rho) := \{ T \in \SST_n(\rho) \mid \text{$T$ is an $\AI$-tableau} \}.
$$
To each $T \in \SST_n^{\AI}(\rho)$, a degree $\deg(T) = (\deg_1(T),\deg_3(T),\ldots,\deg_{2m-1}(T)) \in \Z^m$ is assigned. Then, we can say that $\SST_n^{\AI}(\rho)$ models the $\mathfrak{so}_n$-module $V^{\mathfrak{so}_n}(\rho)$ in the following sense.

\begin{Thm}
Let $\rho \in \Par_m$. Then, we have
$$
\ch V^{\mathfrak{so}_n}(\rho) = \frac{1}{2^m} \sum_{T \in \SST_n^{\AI}(\rho)} \sum_{\sigma_1,\sigma_3,\ldots,\sigma_{2m-1} \in \{+,-\}} y_1^{\sigma_1 \deg_1(T)} y_3^{\sigma_3 \deg_3(T)} \cdots y_{2m-1}^{\sigma_{2m-1} \deg_{2m-1}(T)},
$$
where the left hand-side denotes the character of the $\mathfrak{so}_n$-module $V^{\mathfrak{so}_n}(\rho)$.
\end{Thm}

Let us compare the already known tableau models, some of which can model irreducible representations of half-integer highest weight as well, with ours. King and El-Sharkaway \cite{KES83} introduced the notion of orthogonal tableaux by investigating the branching from $\mathfrak{so}_n$ to $\mathfrak{so}_{n-2}$. A weight is assigned to each orthogonal tableau. Then, they proved that the associated generating functions are irreducible characters of $\mathfrak{so}_n$. Koike and Terada \cite{KT90} introduced another tableau model for $\mathfrak{so}_n$. Their tableaux need $3m$ (resp., $4m$) kinds of letters when $n$ is odd (resp., even). Sundaram \cite{Sun90} (for $\mathfrak{so}_{2m+1}$), Proctor \cite{Pr90}, and Okada \cite{O91} (for $\mathfrak{so}_{2m}$) constructed similar, but more or less simple, tableau models. Kashiwara and Nakashima \cite{KN94} constructed totally different tableau model. Their model is equipped with a crystal structure, which originates from the representation theory of the quantum groups. Our new tableau model is, as a set, different from the other models above, and is much simpler than models of King-El-Sharkaway, Koike-Terada, and Kashiwara-Nakashima; the constraints are only the ordinary semistandardness condition and an easily checked condition on the first two columns. Our character formula is a little bit involved compared to the other models, but it is still easy. What is unique to our model is a new combinatorial structure, which we call the $\AI$-crystal structure. This structure is closely related to the representation theory of the $\imath$quantum group (also known as the quantum symmetric pair coideal subalgebra) of type $\AI$.

An $\imath$quantum group is a certain right coideal subalgebra of a quantum group appearing in the theory of quantum symmetric pairs formulated by G. Letzter \cite{Le99}. Among them, the $\imath$quantum group of type $\AI$, which has appeared in \cite{GK91} earlier than \cite{Le99}, is the subalgebra $\Ui$ of the quantum group $\U = U_q(\frgl_n)$ associated to $\frgl_n$ generated by
$$
B_i := F_i + q\inv E_i K_i\inv, \qu i \in \{1,\ldots,n-1\},
$$
where $E_i,F_i,K_i^{\pm 1}$, $i \in \{ 1,\ldots,n-1 \}$ denote the Chevalley generators of $\U$. Under the classical limit $q \rightarrow 1$, it tends to the universal enveloping algebra $U(\frso_n)$. Here, $\frso_n$ is embedded into $\frgl_n$ as the Lie subalgebra generated by $b_i := f_i + e_i$, $i \in \{ 1,\dots,n-1 \}$.

In \cite{W21b}, the crystal limit $q \rightarrow \infty$ of the action of $B_i$ on a certain class of $\Ui$-modules was defined. It is a linear operator $\Btil_i$. Abstracting properties of $\Btil_i$'s, we introduce the notion of $\AI$-crystals. An $\AI$-crystal is a set $\clB$ equipped with structure maps $\Btil_i : \clB \rightarrow \clB \sqcup \{0\}$ and $\deg_i : \clB \rightarrow \Z$ satisfying certain axioms. Then, it is shown that for each $\rho \in \Par_m$, the set $\SST_n^{\AI}(\rho)$ admits an $\AI$-crystal structure which can be thought of as the crystal limit of the $\Ui$-module $V_q^{\mathfrak{so}_n}(\rho)$ (a $q$-analogue of the $\mathfrak{so}_n$-module $V^{\mathfrak{so}_n}(\rho)$). This is the second achievement in an attempt to generalize the theory of crystal basis to $\imath$quantum groups; the first example was given in \cite{W20} for quasi-split type $\mathrm{AIII}_{2r}$ with asymptotic parameters. No general theory unifying these two crystal theory is known.

Thanks to this $\Ui$-representation theoretic interpretation, it turns out that $\AI$-crystals model not only the $\mathfrak{so}_n$-modules $V^{\mathfrak{so}_n}(\rho)$ but also the tensor product of an $\mathfrak{so}_n$-module and a $\mathfrak{gl}_n$-module, and the branching from $\mathfrak{gl}_n$ to $\mathfrak{so}_n$.

Let $\lm \in \Par_n$ and $\rho \in \Par_m$. Then, $\SST_n^{\AI}(\rho) \otimes \SST_n(\lm) := \SST_n^{\AI}(\rho) \times \SST_n(\lm)$ is equipped with an $\AI$-crystal structure. This structure reflects the $\frso_n$-module structure of $V^{\mathfrak{so}_n}(\rho) \otimes V^{\mathfrak{gl}_n}(\lm)$, which is the classical limit of the $\Ui$-module structure of the corresponding tensor product module (recall that $\Ui$ is a right coideal of $\U$). Mimicking Schensted's insertion algorithm, we introduce an algorithm which tells us how the tensor product $\SST_n^{\AI}(\rho) \otimes \SST_n(\lm)$ decomposes into several copies of $\SST_n^{\AI}(\sigma)$'s, $\sigma \in \Par_m$. Such insertion schemes for other tableau models have been invented in \cite{Sun90,Pr90,O91,Lec03}.

As a special case of the tensor product modules, one can consider $V^{\mathfrak{gl}_n}(\lm) = V^{\mathfrak{so}_n}(\emptyset) \otimes V^{\mathfrak{gl}_n}(\lm)$ because $V^{\mathfrak{so}_n}(\emptyset)$ is the trivial representation. Recall that $\SST_n(\lm)$ models $V^{\mathfrak{gl}_n}(\lm)$. By the argument above, we see that $\SST_n(\lm) = \SST_n^{\AI}(\emptyset) \otimes \SST_n(\lm)$ is equipped with an $\AI$-crystal structure. This structure reflects the $\frso_n$-module structure of $V^{\mathfrak{gl}_n}(\lm)$. To each $T \in \SST_n(\lm)$, an $\AI$-tableau $\std(T)$, called the $P^{\AI}$-symbol of $T$, is assigned. For each $\rho \in \Par_m$, the subset
$$
\{ T \in \SST_n(\lm) \mid \text{the shape of $\std(T)$ is $\rho$} \}
$$
forms an $\AI$-crystal isomorphic to the disjoint union of several copies of $\SST_n^{\AI}(\rho)$. In this way, we obtain an ``irreducible decomposition'' of $\SST_n(\lm)$ as an $\AI$-crystal, which corresponds to a decomposition of $V^{\mathfrak{gl}_n}(\lm)$ as an $\mathfrak{so}_n$-module.

This paper is organized as follows. In Section \ref{Section Combinatorics}, we prepare notation concerning combinatorial objects, and introduce the notion of $\AI$-crystals and $\AI$-tableaux, which play key roles in the construction of our new tableau model. In Section \ref{Section Representation theoretic interpretation}, we give a representation theoretic interpretation to combinatorial objects introduced in the previous section. Also, our main theorem is proved there. In Section \ref{Section RS correspondence}, we develop an insertion scheme for our model. This enables us to understand the tensor product of an $\mathfrak{so}_n$-module and a $\mathfrak{gl}_n$-module and the branching from $\mathfrak{gl}_n$ to $\mathfrak{so}_n$ from an $\AI$-crystal theoretic point of view.

\subsection*{Acknowledgement}
The author would like to thank the referees for careful readings and valuable comments.
He is grateful to Il-Seung Jang for many helpful comments on the draft version of this paper. This work was supported by JSPS KAKENHI Grant Numbers JP20K14286 and JP21J00013.

\subsection*{Data availability}
Data sharing not applicable to this article as no datasets were generated or analysed during the current study.

\subsection*{Notation}
Throughout this paper, we use the following notation:
\begin{itemize}
\item $\Z_{\geq 0}$: the set of nonnegative integers.
\item $\Z_{\ev}$: the set of even integers.
\item $\Z_{\odd}$: the set of odd integers.
\item For $p \in \{ \ev,\odd \}$, $\Z_{\geq 0, p} := \Z_{\geq 0} \cap \Z_p$.
\item For $a,b \in \Z$, $[a,b] := \{ c \in \Z \mid a \leq c \leq b \}$.
\item For $a,b \in \Z$ and $p \in \{ \ev, \odd \}$, $[a,b]_p := [a,b] \cap \Z_p$.
\end{itemize}

\section{Combinatorics}\label{Section Combinatorics}
In this section, we prepare notation concerning combinatorial objects used throughout the paper. After reviewing the theory of crystals of type A, we introduce the notion of $\AI$-crystals and $\AI$-tableaux, which play key roles in the construction of our new tableau model for finite-dimensional $\mathfrak{so}_n$-modules of integer highest weight.

\subsection{Tableaux}
A partition is a non-increasing sequence $\lm = (\lm_1,\ldots,\lm_l)$ of positive integers; $l$ is referred to as the length of $\lm$, and is denoted by $\ell(\lm)$. The empty sequence $\emptyset$ is the unique partition of length $0$. The size of a partition $\lm = (\lm_1,\ldots,\lm_l)$ is defined to be $\sum_{i=1}^l \lm_i$, and is denoted by $|\lm|$. Let $\Par_l$ denote the set of partitions of length not greater than $l$, and $\Par := \bigcup_{l \geq 0} \Par_l$ the set of partitions.

The Young diagram of shape $\lm \in \Par$ is the set
$$
D(\lm) := \{ (i,j) \mid i \in [1,\ell(\lm)], \ j \in [1,\lm_i] \}.
$$
For $\lm,\mu \in \Par$, we write $\lm \triangleleft \mu$ if $D(\lm) \subset D(\mu)$ and $|\mu|-|\lm| = 1$.

Let $A$ be a set. A Young tableau of shape $\lm \in \Par$ in alphabet $A$ is a map
$$
T : D(\lm) \rightarrow A.
$$
Unless otherwise stated, we always fix $n \geq 2$ and take $A = [1,n]$ as the alphabet. When considering Young tableaux, each element of the alphabet is referred to as a letter. For a Young tableau $T$ of shape $\lm \in \Par$ and $(i,j) \in D(\lm)$, the letter $T(i,j)$ is referred to as the $(i,j)$-th entry of $T$. The partition $\lm$ is referred to as the shape of $T$, and is denoted by $\sh(T)$. The size of $T$ is defined to be $|\lm|$, and is denoted by $|T|$. For a subset $A \subset [1,n]$, let $T|_A$ denote the map
$$
T|_A : \{ (i,j) \in D(\lm) \mid T(i,j) \in A \} \rightarrow A;\ (i,j) \mapsto T(i,j).
$$

Given $n \geq d_1 \geq d_2 \geq \cdots \geq d_l > 0$ and $C_j = (c_{1,j},c_{2,j},\ldots,c_{d_j,j}) \in [1,n]^{d_j}$, let $C_1 C_2 \cdots C_l$ denote the Young tableau given by
$$
(C_1 C_2 \cdots C_l)(i,j) := c_{i,j}.
$$

A sequence of letters is referred to as a word. Let $\clW$ denote the set of words, i.e.,
$$
\clW := \bigsqcup_{d \geq 0} [1,n]^d.
$$
For two words $w_1,w_2 \in \clW$, let $w_1 * w_2 \in \clW$ denote the concatenation of them. The column reading of a Young tableau $T$ is a word $\CR(T) \in \clW$ defined to be
$$
(t_{d_1,1}, t_{d_1-1,1}, \ldots, t_{1,1}) * (t_{d_2,2}, t_{d_2-1,2}, \ldots, t_{1,2}) * \cdots * (t_{d_{\lm_1},\lm_1}, t_{d_{\lm_1}-1,\lm_1}, \ldots, t_{1,\lm_1}),
$$
where
$$
d_j := \max\{ i \mid \lm_i \geq j \}
$$
denotes the length of the $j$-th column of $D(\lm)$.

\begin{ex}\normalfont\label{Example of column reading}
Let $\lm = (4,2,1)$, and
$$
T = \ytableausetup{centertableaux}
\begin{ytableau}
1 & 2 & 3 & 3\\
2 & 3 \\
4
\end{ytableau}.
$$
Then, its column reading is
$$
\CR(T) = (4,2,1) * (3,2) * (3) * (3) = (4,2,1,3,2,3,3).
$$
\end{ex}

A Young tableau is said to be semistandard if its entries weakly increase along the rows from left to right, and strongly increase along the columns from top to bottom. Let $\SST_n(\lm)$ denote the set of semistandard Young tableaux of shape $\lm \in \Par$. Note that $\SST_n(\lm) = \emptyset$ unless $\ell(\lm) \leq n$.

A semistandard Young tableau of shape $\lm \in \Par$ in alphabet $[1,N]$ for some $N > 0$ is said to be standard if the assignment $(i,j) \mapsto T(i,j)$ is injective. Let $\STab_N(\lm)$ denote the set of standard Young tableaux of shape $\lm$ in alphabet $[1,N]$.

For a semistandard Young tableau $T \in \SST_n(\lm)$ and a letter $l \in [1,n]$, let $(T \leftarrow l)$ denote the semistandard Young tableau obtained by Schensted's row insertion algorithm (see e.g., \cite[Chapter 7.1]{BS17} for a precise definition).

\begin{ex}\normalfont
Let $T$ be as in Example \ref{Example of column reading}. Then, we have
$$
(T \leftarrow 1) = \ytableausetup{centertableaux}
\begin{ytableau}
1 & 1 & 3 & 3 \\
2 & 2 \\
3 \\
4
\end{ytableau}, \ (T \leftarrow 2) = \ytableausetup{centertableaux}
\begin{ytableau}
1 & 2 & 2 & 3 \\
2 & 3 & 3 \\
4
\end{ytableau}, \ (T \leftarrow 3) = \ytableausetup{centertableaux}
\begin{ytableau}
1 & 2 & 3 & 3 & 3 \\
2 & 3 \\
4
\end{ytableau}.
$$
\end{ex}

Let us recall the Robinson-Schensted correspondence. Fix a word $w = (w_1,\ldots,w_d) \in \clW$. For each $k \in [0,d]$, define a semistandard Young tableau $P^k$ inductively by
$$
P^k := \begin{cases}
\emptyset \qu & \IF k = 0, \\
(P^{k-1} \leftarrow w_k) \qu & \IF k > 0.
\end{cases}
$$
Let $\lm^k $ denote the shape of $P^k$. Note that we have $\lm^{k-1} \triangleleft \lm^k$. Define a standard Young tableau $Q^k \in \STab_k(\lm^k)$ inductively by $Q^0 := \emptyset$, and
$$
Q^k(i,j) := \begin{cases}
Q^{k-1}(i,j) \qu & \IF (i,j) \in D(\lm^{k-1}), \\
k \qu & \IF (i,j) \in D(\lm^k) \setminus D(\lm^{k-1})
\end{cases}
$$
for $k \in [1,d]$. The tableaux $P^d$ and $Q^d$, and the partition $\lm^d$ are referred to as the P-symbol, the Q-symbol, and the shape of $w$, and are denoted by $P(w)$, $Q(w)$, and $\sh(w)$, respectively. The assignment
$$
\RS : \clW \rightarrow \bigsqcup_{\lm \in \Par_n} \SST_n(\lm) \times \STab_{|\lm|}(\lm);\ w \mapsto (P(w),Q(w))
$$
is called the Robinson-Schensted correspondence. As is well-known, $\RS$ is bijective.

\begin{ex}\normalfont
Let $w = (4,2,3,1,3,2)$. Then, $P^k$ and $Q^k$, $k \in [0,6]$ are as follows:
\begin{align}
\begin{split}
&P^0 = \emptyset,\ \ytableausetup{centertableaux}
\begin{ytableau}
4
\end{ytableau},\ \ytableausetup{centertableaux}
\begin{ytableau}
2 \\
4
\end{ytableau},\ \ytableausetup{centertableaux}
\begin{ytableau}
2 & 3 \\
4
\end{ytableau},\ \ytableausetup{centertableaux}
\begin{ytableau}
1 & 3 \\
2 \\
4
\end{ytableau},\ \ytableausetup{centertableaux}
\begin{ytableau}
1 & 3 & 3 \\
2 \\
4
\end{ytableau},\ \ytableausetup{centertableaux}
\begin{ytableau}
1 & 2 & 3 \\
2 & 3 \\
4
\end{ytableau} = P^6 = P(w), \\
&Q^0 = \emptyset,\ \ytableausetup{centertableaux}
\begin{ytableau}
1
\end{ytableau},\ \ytableausetup{centertableaux}
\begin{ytableau}
1 \\
2
\end{ytableau},\ \ytableausetup{centertableaux}
\begin{ytableau}
1 & 3 \\
2
\end{ytableau},\ \ytableausetup{centertableaux}
\begin{ytableau}
1 & 3 \\
2 \\
4
\end{ytableau},\ \ytableausetup{centertableaux}
\begin{ytableau}
1 & 3 & 5 \\
2 \\
4
\end{ytableau},\ \ytableausetup{centertableaux}
\begin{ytableau}
1 & 3 & 5 \\
2 & 6 \\
4
\end{ytableau} = Q^6 = Q(w).
\end{split} \nonumber
\end{align}
\end{ex}

\subsection{$\frgl$-Crystal}
Combinatorics which have been introduced so far are closely related to the representation theory of the general linear algebra $\frgl_n = \frgl_n(\C)$ via the theory of crystals.

A $\frgl_n$-crystal (we omit the subscript ``$n$'' when there is no confusion) is a set $\clB$ equipped with structure maps
$$
\Etil_i,\Ftil_i : \clB \rightarrow \clB \sqcup \{0\}, \qu \vep_i,\vphi_i : \clB \rightarrow \Z \sqcup \{-\infty\}, \qu i \in [1,n-1],
$$
where $0$ is a formal symbol, and
$$
\wt : \clB \rightarrow \Z^n
$$
satisfying certain axioms (see e.g., \cite[Definition 2.13]{BS17}). Every $\frgl$-crystal appearing in this paper is a Stembridge crystal in the sense of \cite[Chapter 4.2]{BS17}. Among the axioms of (Stembridge) crystals, what are particularly important for us are the following: Let $b,b' \in \clB$, $i,j \in [1,n-1]$. Then, the following hold.
\begin{enumerate}
\item $\Ftil_i b = b'$ if and only if $b = \Etil_i b'$.
\item $\vep_i(b) := \max\{ k \mid \Etil_i^k b \neq 0 \}$.
\item $\vphi_i(b) := \max\{ k \mid \Ftil_i^k b \neq 0 \}$.
\item If $|i-j| > 1$ and $\Etil_i b \neq 0$, then $\vep_j(\Etil_i b) = \vep_j(b)$, $\vphi_j(\Etil_i b) = \vphi_j(b)$, and $\Etil_j \Etil_i b = \Etil_i \Etil_j b$.
\item If $|i-j| = 1$ and $\Etil_i b \neq 0$, then either $\vep_j(\Etil_i b) = \vep_j(b)$ and $\vphi_j(\Etil_i b) = \vphi_j(b)-1$, or $\vep_j(\Etil_i b) = \vep_j(b)+1$ and $\vphi_j(\Etil_j b) = \vphi_j(b)$.
\end{enumerate}
Here, we set $\Etil_i 0 = 0$ and $\Ftil_i 0 = 0$.

Let $\clB_1,\clB_2$ be $\frgl$-crystals. A morphism $\psi : \clB_1 \rightarrow \clB_2$ of $\frgl$-crystals is a map $\psi : \clB_1 \sqcup \{0\} \rightarrow \clB_2 \sqcup \{0\}$ such that
\begin{align}
\begin{split}
&\psi(0) = 0,\ \psi(\Etil_i b) = \Etil_i \psi(b), \  \psi(\Ftil_i b) = \Ftil_i \psi(b), \\
&\vep_i(\psi(b')) = \vep_i(b'), \  \vphi_i(\psi(b')) = \vphi_i(b'), \  \wt(\psi(b')) = \wt(b')
\end{split} \nonumber
\end{align}
for all $b,b' \in \clB_1$ and $i \in [1,n-1]$ such that $\psi(b') \neq 0$. A morphism of $\frgl$-crystals is said to be an isomorphism if the underlying map $\clB_1 \sqcup \{0\} \rightarrow \clB_2 \sqcup \{0\}$ is bijective.

The character $\ch_{\frgl} \clB$ of $\clB$ is a Laurent polynomial in $n$ variables $x_1,\ldots,x_n$ given by
$$
\ch_{\frgl} \clB := \sum_{b \in \clB} \bfx^{\wt(b)} \in \Z[x_1^{\pm 1},\ldots,x_n^{\pm 1}],
$$
where $\bfx^{(a_1,\ldots,a_n)} := x_1^{a_1} \cdots x_n^{a_n}$.

The crystal graph of $\clB$ is a colored directed graph defined as follows. The vertex set is $\clB$. For each $b,b' \in \clB$ and $i \in [1,n-1]$, there is an $i$-colored arrow from $b$ to $b'$ if and only if $b' = \Ftil_i(b)$.

\begin{ex}\normalfont
The set $\SST_n(1)$ is equipped with a $\frgl$-crystal structure as follows:
$$
\wt\left( \Cbox{j} \right) := \eps_j, \qu \Etil_i \Cbox{j} := \begin{cases}
\Cbox{\scriptstyle j-1} & \IF j = i+1, \\
0 & \OW,
\end{cases} \qu \Ftil_i \Cbox{j} := \begin{cases}
\Cbox{\scriptstyle j+1} & \IF j = i, \\
0 & \OW,
\end{cases}
$$
where
$$
\eps_j := (\overbrace{0,\ldots,0}^{j-1},1,0,\ldots,0) \in \Z^n.
$$
Then, we have
$$
\ch_{\frgl} \SST_n(1) = x_1 + x_2 + \cdots + x_n.
$$
The crystal graph of $\SST_n(1)$ is
$$
\Cbox{1} \xrightarrow[]{1} \Cbox{2} \xrightarrow[]{2} \Cbox{3} \xrightarrow[]{3} \cdots \xrightarrow[]{n-1} \Cbox{n}.
$$
\end{ex}

Let $\clB_1,\clB_2$ be $\frgl$-crystals. Then, $\clB_1 \otimes \clB_2 := \clB_1 \times \clB_2$ is equipped with a $\frgl$-crystal structure as follows (cf. \cite[Section 2.3]{BS17}): It is customary to denote $(b_1,b_2) \in \clB_1 \otimes \clB_2$ by $b_1 \otimes b_2$. For each $i \in [1,n-1]$, we set
\begin{align}
\begin{split}
&\wt(b_1 \otimes b_2) := \wt(b_1) + \wt(b_2), \\
&\Ftil_i(b_1 \otimes b_2) := \begin{cases}
\Ftil_i(b_1) \otimes b_2 \qu & \IF \vep_i(b_1) \geq \vphi_i(b_2), \\
b_1 \otimes \Ftil_i(b_2) \qu & \IF \vep_i(b_1) < \vphi_i(b_2),
\end{cases} \\
&\Etil_i(b_1 \otimes b_2) := \begin{cases}
\Etil_i(b_1) \otimes b_2 \qu & \IF \vep_i(b_1) > \vphi_i(b_2), \\
b_1 \otimes \Etil_i(b_2) \qu & \IF \vep_i(b_1) \leq \vphi_i(b_2).
\end{cases}
\end{split} \nonumber
\end{align}
Here, we set $0 \otimes b_2 = b_1 \otimes 0 = 0$.
This definition is consistent with \cite{BS17}, and opposite to \cite{Ka90}. The tensor product of $\frgl$-crystals is associative;
$$
(\clB_1 \otimes \clB_2) \otimes \clB_3 = \clB_1 \otimes (\clB_2 \otimes \clB_3).
$$

For each $d \geq 0$, we can construct a $\frgl$-crystal $\SST_n(1)^{\otimes d}$ inductively by
$$
\SST_n(1)^{\otimes d} := \SST_n(1)^{\otimes d-1} \otimes \SST_n(1),
$$
where, $\SST_n(1)^{\otimes 0} = \{ \emptyset \}$ with $\frgl$-crystal structure given by
$$
\wt(\emptyset) = (0,\ldots,0), \qu \Etil_i(\emptyset) = 0 = \Ftil_i(\emptyset).
$$
By associativity of the tensor product of $\frgl$-crystals, we see that
$$
\SST_n(1)^{\otimes d_1} \otimes \SST_n(1)^{\otimes d_2} = \SST_n(1)^{\otimes d_1+d_2}.
$$

Recall that $\clW = \bigsqcup_{d \geq 0} [1,n]^d$ denotes the set of words. By identifying each word $(w_1,w_2,\ldots,w_d) \in \clW$ with
$$
\Cbox{w_1} \otimes \Cbox{w_2} \otimes \cdots \otimes \Cbox{w_d} \in \SST_n(1)^{\otimes d},
$$
one can equip $\clW$ with a $\frgl$-crystal structure. Then, the concatenation of words is identical to the tensor product of $\frgl$-crystals.

Let $\lm \in \Par_n$. Then, $\SST_n(\lm)$ is equipped with a $\frgl$-crystal structure as follows: For each $T \in \SST_n(\lm)$ and $i \in [1,n-1]$, we set
$$
\wt(T) := \wt(\CR(T)), \qu \Etil_i T := P(\Etil_i \CR(T)), \qu \Ftil_i T := P(\Ftil_i \CR(T)),
$$
where we set $P(0) = 0$. For example, the crystal graph of $\SST_3(2,1)$ is as follows.
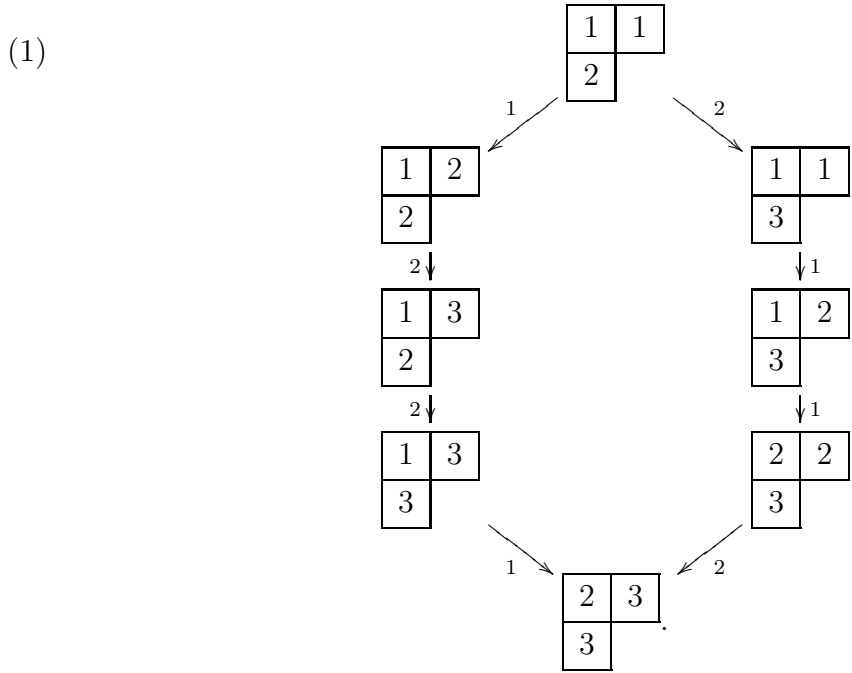
\begin{align}\label{eq: crystal graph of 2,1}
  \xymatrix@R=10pt{
    & \text{$\ytableausetup{centertableaux}
    \begin{ytableau}
      1 & 1 \\
      2
    \end{ytableau}$} \ar[dl]_1 \ar[dr]^2 & \\
    \text{$\ytableausetup{centertableaux}
    \begin{ytableau}
      1 & 2 \\
      2
    \end{ytableau}$} \ar[d]_2 & & \text{$\ytableausetup{centertableaux}
    \begin{ytableau}
      1 & 1 \\
      3
    \end{ytableau}$} \ar[d]^1 \\
    \text{$\ytableausetup{centertableaux}
    \begin{ytableau}
      1 & 3 \\
      2
    \end{ytableau}$} \ar[d]_2 & & \text{$\ytableausetup{centertableaux}
    \begin{ytableau}
      1 & 2 \\
      3
    \end{ytableau}$} \ar[d]^1 \\
    \text{$\ytableausetup{centertableaux}
    \begin{ytableau}
      1 & 3 \\
      3
    \end{ytableau}$} \ar[dr]_1 & & \text{$\ytableausetup{centertableaux}
    \begin{ytableau}
      2 & 2 \\
      3
    \end{ytableau}$} \ar[dl]^2 \\
    & \text{$\ytableausetup{centertableaux}
    \begin{ytableau}
      2 & 3 \\
      3
    \end{ytableau}$}. &
  }
\end{align}
Then, the Robinson-Schensted correspondence
$$
\RS : \clW \rightarrow \bigsqcup_{\lm \in \Par_n} \SST_n(\lm) \times \STab_{|\lm|}(\lm);\ w \mapsto (P(w),Q(w))
$$
is an isomorphism of $\frgl$-crystals (cf. \cite[Section 8.3]{BS17}), where the $\frgl$-crystal structure of $\SST_n(\lm) \times \STab_{|\lm|}(\lm)$ is given by
$$
\wt(P,Q) := \wt(P), \qu \Etil_i(P,Q) := (\Etil_i P,Q), \qu \Ftil_i(P,Q) := (\Ftil_i P,Q).
$$

Given partitions $\lm,\mu \in \Par_n$ and tableaux $T \in \SST_n(\lm)$, $S \in \SST_n(\mu)$, define the $P$-symbol $P(T \otimes S)$ of $T \otimes S$ to be the $P$-symbol of $\CR(T) * \CR(S)$. For example, we have
$$
P\left( \ytableausetup{centertableaux}
\begin{ytableau}
1 & 3 \\
3
\end{ytableau} \otimes \ytableausetup{centertableaux}
\begin{ytableau}
1 & 2 \\
2 & 3 \\
4
\end{ytableau} \right) = P((3,1,3)*(4,2,1,3,2)) = \ytableausetup{centertableaux}
\begin{ytableau}
1 & 1 & 2 \\
2 & 3 & 3 \\
3 & 4
\end{ytableau}.
$$

For a later use, we put here an easy observation.

\begin{lem}\label{observation of C1C2}
Let $k,l \in [1,n]$, $1 \leq i_1 < \cdots < i_k \leq n$, and $1 \leq j_1 < \cdots < j_l \leq n$, and set
$$
C_1 := P(i_k,\ldots,i_1) = \ytableausetup{centertableaux}
\begin{ytableau}
i_1 \\
\vdots \\
i_k
\end{ytableau}, \qu C_2 := P(j_l,\ldots,j_1) = \ytableausetup{centertableaux}
\begin{ytableau}
j_1 \\
\vdots \\
j_l
\end{ytableau},
$$
and $\lm := \sh(P(C_1 \otimes C_2))$. Then, we have $\ell(\lm) \geq k$. Furthermore, the following are equivalent:
\begin{enumerate}
\item $\ell(\lm) = k$.
\item $k \geq l$ and $i_r \leq j_r$ for all $r \in [1,l]$.
\item $P(C_1 \otimes C_2) = C_1C_2$.
\end{enumerate}
\end{lem}

Let $\clB$ be a $\frgl$-crystal and $b,b_1,b_2 \in \clB$. We say that $b_1$ and $b_2$ are connected, or $b_2$ is connected to $b_1$ if we have
$$
b_2 = \Xtil_{i_1} \Xtil_{i_2} \cdots \Xtil_{i_r} b_1
$$
for some $\Xtil_{i_1},\ldots,\Xtil_{i_r} \in \{\Etil_i,\Ftil_i \mid i \in [1,n-1] \}$. The connected component $C(b)$ of $\clB$ containing $b$ is defined by
$$
C(b) := \{ b' \in \clB \mid \text{$b'$ is connected to $b$} \}.
$$
We say that $\clB$ is connected if we have $\clB = C(b)$ for some $b \in \clB$.

Let $w,w_1,w_2 \in \clW$. As is well-known, $w_1$ and $w_2$ are connected if and only if their $Q$-symbols coincide. Hence, we have
$$
C(w) = \{ w' \in \clW \mid Q(w') = Q(w) \}.
$$
Under the Robinson-Schensted correspondence, this connected component corresponds to $\SST_n(\lm) \times \{ Q(w) \} \simeq \SST_n(\lm)$.

\subsection{$\AI$-crystal}\label{Subsection AI-crystal}
In this subsection, we introduce the notion of $\AI_{n-1}$-crystals (we omit the subscript ``$n-1$'' when there is no confusion). From now on, we assume that $n \geq 3$, and set
\begin{align}\label{eq: def of m}
  m := \begin{cases}
    \frac{n}{2} \qu & \IF n \in \Z_{\ev}, \\
    \frac{n-1}{2} \qu & \IF n \in \Z_{\odd}.
  \end{cases}
\end{align}

\begin{defi}\label{definition of AI-crystals}\normalfont
An $\AI$-crystal is a set $\clB$ equipped with structure maps $\Btil_i : \clB \rightarrow \clB \sqcup \{0\}$ and $\deg_i : \clB \rightarrow \Z_{\geq 0}$, $i \in [1,n-1]$ satisfying the following axioms: Let $b \in \clB$ and $i,j \in [1,n-1]$.
\begin{enumerate}
\item\label{definition of AI-crystals 1} If $\Btil_i b \neq 0$, then $\deg_i(\Btil_i b) = \deg_i(b)$ and $\Btil_i^2 b = b$.
\item\label{definition of AI-crystals 2} If $\Btil_i b \neq 0$ and $|i-j| = 1$, then $\deg_j(\Btil_i b) - \deg_j(b) \in \{ 1,-1\}$.
\item\label{definition of AI-crystals 3} If $\Btil_i b \neq 0$ and $|i-j| > 1$, then $\deg_{j}(\Btil_{i}b) = \deg_{j}(b)$.
\end{enumerate}
\end{defi}

\begin{defi}\normalfont
Let $\clB_1,\clB_2$ be $\AI$-crystals. A morphism $\psi : \clB_1 \rightarrow \clB_2$ of $\AI$-crystals is a map $\psi : \clB_1 \sqcup \{0\} \rightarrow \clB_2 \sqcup \{0\}$ such that
$$
\psi(0) = 0,\ \psi(\Btil_i b) = \Btil_i \psi(b), \  \deg_i(\psi(b)) = \deg_i(b)
$$
for all $b \in \clB_1$ and $i \in [1,n-1]$. A morphism of $\AI$-crystals is said to be an isomorphism if the underlying map is bijective.
\end{defi}

\begin{defi}\normalfont
Let $\clB$ be an $\AI$-crystal. The $\AI$-crystal graph of $\clB$ is a colored (non-directed) graph defined as follows. The vertex set is $\clB$. For each $b,b' \in \clB$ and $i \in [1,n-1]$, there is an $i$-colored edge between $b$ and $b'$ if and only if $b' = \Btil_i(b)$.
\end{defi}

\begin{defi}\normalfont
Let $\clB$ be an $\AI$-crystal and $b,b_1,b_2 \in \clB$. We say that $b_1$ and $b_2$ are connected, or $b_2$ is connected to $b_1$ if we have
$$
b_2 = \Btil_{i_1} \Btil_{i_2} \cdots \Btil_{i_r} b_1
$$
for some $i_1,\ldots,i_r \in [1,n-1]$. We call
$$
C^{\AI}(b) := \{ b' \in \clB \mid \text{$b'$ is connected to $b$} \}
$$
the connected component $C^{\AI}(b)$ of $\clB$ containing $b$. We say that $\clB$ is connected if we have $\clB = C^{\AI}(b)$ for some $b \in \clB$.
\end{defi}

\begin{ex}\normalfont
$\SST_n(1)$ is equipped with an $\AI$-crystal structure such that
$$
\deg_i\left( \Cbox{j} \right) = \begin{cases}
1 \qu & \IF j \in \{ i,i+1 \}, \\
0 \qu & \OW
\end{cases}, \qu \Btil_i \Cbox{j} = \begin{cases}
\Cbox{\scriptstyle j-1} \qu & \IF j = i+1, \\
\\
\Cbox{\scriptstyle j+1} \qu & \IF j = i, \\
\\
0 \qu & \OW.
\end{cases}
$$
Then, its $\AI$-crystal graph is
$$
\xymatrix{
\Cbox{1} \ar@{-}[r]^1 & \Cbox{2} \ar@{-}[r]^2 & \Cbox{3} \ar@{-}[r]^3 & \cdots \ar@{-}[r]^{\scriptstyle n-1} & \Cbox{n}.
}
$$
\end{ex}

\begin{prop}\label{AI-crystal tensor gl-crystal}
Let $\clB_1$ be an $\AI$-crystal and $\clB_2$ a $\frgl$-crystal. Then, $\clB_1 \otimes \clB_2 := \clB_1 \times \clB_2$ is equipped with an $\AI$-crystal structure as follows:
\begin{align}
\begin{split}
&\deg_i(b_1 \otimes b_2) = \begin{cases}
\deg_i(b_1) - \vphi_i(b_2) + \vep_i(b_2) \qu & \IF \deg_i(b_1) > \vphi_i(b_2), \\
\vep_i(b_2) \qu & \IF \vphi_i(b_2) - \deg_i(b_1) \in \Z_{\geq 0, \ev}, \\
\vep_i(b_2)+1 \qu & \IF \vphi_i(b_2) - \deg_i(b_1) \in \Z_{\geq 0, \odd}.
\end{cases} \\
&\Btil_i(b_1 \otimes b_2) = \begin{cases}
\Btil_i b_1 \otimes b_2 \qu & \IF \deg_i(b_1) > \vphi_i(b_2), \\
b_1 \otimes \Etil_i b_2 \qu & \IF \vphi_i(b_2) - \deg_i(b_1) \in \Z_{\geq 0, \ev}, \\
b_1 \otimes \Ftil_i b_2 \qu & \IF \vphi_i(b_2) - \deg_i(b_1) \in \Z_{\geq 0, \odd}.
\end{cases}
\end{split} \nonumber
\end{align}
\end{prop}

\begin{proof}
Let us verify that $\clB_1 \otimes \clB_2$ satisfies the axioms of $\AI$-crystal. During the proof, we use axioms of $\frgl$-crystal and $\AI$-crystal without mentioning one by one.
\begin{enumerate}

\item Let $b_1 \in \clB_1, b_2 \in \clB_2$ and $i \in [1,n-1]$ be such that $\Btil_i(b_1 \otimes b_2) \neq 0$. We show that $\deg_i(\Btil_i(b_1 \otimes b_2)) = \deg_i(b_1 \otimes b_2)$ and $\Btil_i^2(b_1 \otimes b_2) = b_1 \otimes b_2$.

First, suppose that $\deg_i(b_1) > \vphi_i(b_2)$. In this case, we have
$$
\Btil_i(b_1 \otimes b_2) = \Btil_i b_1 \otimes b_2.
$$
Since $\deg_i(\Btil_i b_1) = \deg_i(b_1) > \vphi_i(b_2)$, we obtain
\begin{align}
\begin{split}
\deg_i(\Btil_i b_1 \otimes b_2) &= \deg_i(\Btil_i b_1) - \vphi_i(b_2) + \vep_i(b_2) \\
&= \deg_i(b_1) - \vphi_i(b_2) + \vep_i(b_2) = \deg_i(b_1 \otimes b_2),
\end{split} \nonumber
\end{align}
and
$$
\Btil_i(\Btil_i b_1 \otimes b_2) = \Btil_i^2 b_1 \otimes b_2 = b_1 \otimes b_2,
$$
as desired.

Next, suppose that $\vphi_i(b_2)-\deg_i(b_1) \in \Z_{\geq 0, \ev}$. In this case, we have
$$
\Btil_i(b_1 \otimes b_2) = b_1 \otimes \Etil_i b_2.
$$
Since $\vphi_i(\Etil_i b_2) = \vphi_i(b_2) + 1$, we have $\vphi_i(\Etil_i b_2) - \deg_i(b_1) \in \Z_{\geq 0, \odd}$. Hence, we obtain
$$
\deg_i(b_1 \otimes \Etil_i b_2) = \vep_i(\Etil_i b_2)+1 = \vep_i(b_2) = \deg_i(b_1 \otimes b_2),
$$
and 
$$
\Btil_i(b_1 \otimes \Etil_i b_2) = b_1 \otimes \Ftil_i \Etil_i b_2 = b_1 \otimes b_2,
$$
as desired.

Finally, suppose that $\vphi_i(b_2)-\deg_i(b_1) \in \Z_{\geq 0, \odd}$. In this case, we have
$$
\Btil_i(b_1 \otimes b_2) = b_1 \otimes \Ftil_i b_2.
$$
Since $\vphi_i(\Ftil_i b_2) = \vphi_i(b_2) - 1$, we have $\vphi_i(\Ftil_i b_2) - \deg_i(b_1) \in \Z_{\geq 0, \ev}$. Hence, we obtain
$$
\deg_i(b_1 \otimes \Ftil_i b_2) = \vep_i(\Ftil_i b_2) = \vep_i(b_2) + 1 = \deg_i(b_1 \otimes b_2),
$$
and
$$
\Btil_i(b_1 \otimes \Ftil_i b_2) = b_1 \otimes \Etil_i \Ftil_i b_2 = b_1 \otimes b_2,
$$
as desired.

\item Let $b_1 \in \clB_1$, $b_2 \in \clB_2$, $i,j \in [1,n-1]$ be such that $|i-j|=1$ and $\Btil_i (b_1 \otimes b_2) \neq 0$. We show that $\deg_j(\Btil_i(b_1 \otimes b_2)) - \deg_j(b_1 \otimes b_2) \in \{1,-1\}$.

Let us write $\Btil_i(b_1 \otimes b_2) = b'_1 \otimes b'_2$ for some $b'_1 \in \{ b_1,\Btil_i b_1 \} \setminus \{0\}$ and $b'_2 \in \{ b_2,\Etil_i b_2, \Ftil_i b_2 \} \setminus \{0\}$. Then, exactly one of the following holds:
\begin{itemize}
\item $\deg_j(b'_1)-\deg_j(b_1) = 1$, $\vphi_j(b'_2)-\vphi_j(b_2) = 0$, and $\vep_j(b'_2)-\vep_j(b_2) = 0$.
\item $\deg_j(b'_1)-\deg_j(b_1) = -1$, $\vphi_j(b'_2)-\vphi_j(b_2) = 0$, and $\vep_j(b'_2)-\vep_j(b_2) = 0$.
\item $\deg_j(b'_1)-\deg_j(b_1) = 0$, $\vphi_j(b'_2)-\vphi_j(b_2) = -1$, and $\vep_j(b'_2)-\vep_j(b_2) = 0$.
\item $\deg_j(b'_1)-\deg_j(b_1) = 0$, $\vphi_j(b'_2)-\vphi_j(b_2) = 0$, and $\vep_j(b'_2)-\vep_j(b_2) = 1$.
\item $\deg_j(b'_1)-\deg_j(b_1) = 0$, $\vphi_j(b'_2)-\vphi_j(b_2) = 1$, and $\vep_j(b'_2)-\vep_j(b_2) = 0$.
\item $\deg_j(b'_1)-\deg_j(b_1) = 0$, $\vphi_j(b'_2)-\vphi_j(b_2) = 0$, and $\vep_j(b'_2)-\vep_j(b_2) = -1$.
\end{itemize}
Therefore, we have either
$$
\deg_j(b'_1) - \vphi_j(b'_2) \in \{ \deg_j(b_1)-\vphi_j(b_2) \pm 1 \} \AND \vep_j(b'_2) = \vep_j(b_2)
$$
or
$$
\deg_j(b'_1) - \vphi_j(b'_2) = \deg_j(b_1)-\vphi_j(b_2) \AND \vep_j(b'_2) \in \{ \vep_j(b_2) \pm 1 \}.
$$

First, suppose that $\deg_j(b'_1)-\vphi_j(b'_2) = \deg_j(b_1)-\vphi_j(b_2)+1$. Then, we compute as
\begin{align}
\begin{split}
\deg_j&(b'_1 \otimes b'_2) \\
&=\begin{cases}
\deg_j(b'_1)-\vphi_j(b'_2)+\vep_j(b'_2) \qu & \IF \deg_j(b'_1) > \vphi_j(b'_2), \\
\vep_j(b'_2) \qu & \IF \vphi_j(b'_2) - \deg_j(b'_1) \in \Z_{\geq 0, \ev}, \\
\vep_j(b'_2)+1 \qu & \IF \vphi_j(b'_2) - \deg_j(b'_1) \in \Z_{\geq 0, \odd}
\end{cases} \\
&=\begin{cases}
\deg_j(b_1)-\vphi_j(b_2)+1+\vep_j(b_2) \qu & \IF \deg_j(b_1) \geq \vphi_j(b_2), \\
\vep_j(b_2) \qu & \IF \vphi_j(b_2) - \deg_j(b_1) \in \Z_{\geq 0, \odd}, \\
\vep_j(b_2)+1 \qu & \IF \vphi_j(b_2) - \deg_j(b_1) \in \Z_{\geq 0, \ev} \setminus \{0\}
\end{cases} \\
&=\begin{cases}
\deg_j(b_1 \otimes b_2) + 1 \qu & \IF \deg_j(b_1) \geq \vphi_j(b_2), \\
\deg_j(b_1 \otimes b_2)-1 \qu & \IF \vphi_j(b_2) - \deg_j(b_1) \in \Z_{\geq 0, \odd}, \\
\deg_j(b_1 \otimes b_2)+1 \qu & \IF \vphi_j(b_2) - \deg_j(b_1) \in \Z_{\geq 0, \ev} \setminus \{0\}
\end{cases} \\
&\in \{ \deg_j(b_1 \otimes b_2) \pm 1 \}.
\end{split} \nonumber
\end{align}

Next, suppose that $\deg_j(b'_1)-\vphi_j(b'_2) = \deg_j(b_1)-\vphi_j(b_2)-1$. Then, we compute as
\begin{align}
\begin{split}
\deg_j&(b'_1 \otimes b'_2) \\
&=\begin{cases}
\deg_j(b'_1)-\vphi_j(b'_2)+\vep_j(b'_2) \qu & \IF \deg_j(b'_1) > \vphi_j(b'_2), \\
\vep_j(b'_2) \qu & \IF \vphi_j(b'_2) - \deg_j(b'_1) \in \Z_{\geq 0, \ev}, \\
\vep_j(b'_2)+1 \qu & \IF \vphi_j(b'_2) - \deg_j(b'_1) \in \Z_{\geq 0, \odd}
\end{cases} \\
&=\begin{cases}
\deg_j(b_1)-\vphi_j(b_2)-1+\vep_j(b_2) \qu & \IF \deg_j(b_1) > \vphi_j(b_2)+1, \\
\vep_j(b_2) \qu & \IF \vphi_j(b_2) - \deg_j(b_1) \in \Z_{\geq 0, \odd} \sqcup \{-1\}, \\
\vep_j(b_2)+1 \qu & \IF \vphi_j(b_2) - \deg_j(b_1) \in \Z_{\geq 0, \ev}
\end{cases} \\
&=\begin{cases}
\deg_j(b_1 \otimes b_2) - 1 \qu & \IF \deg_j(b_1) > \vphi_j(b_2)+1, \\
\deg_j(b_1 \otimes b_2)-1 \qu & \IF \vphi_j(b_2) - \deg_j(b_1) \in \Z_{\geq 0, \odd} \sqcup \{-1\}, \\
\deg_j(b_1 \otimes b_2)+1 \qu & \IF \vphi_j(b_2) - \deg_j(b_1) \in \Z_{\geq 0, \ev}
\end{cases} \\
&\in \{ \deg_j(b_1 \otimes b_2) \pm 1 \}.
\end{split} \nonumber
\end{align}

Finally, suppose that $\deg_j(b'_1)-\vphi_j(b'_2) = \deg_j(b_1)-\vphi_j(b_2)$. Setting $a := \vep_j(b'_2)-\vep_j(b_2) \in \{\pm 1\}$, we compute as
\begin{align}
\begin{split}
\deg_j&(b'_1 \otimes b'_2) \\
&=\begin{cases}
\deg_j(b'_1)-\vphi_j(b'_2)+\vep_j(b'_2) \qu & \IF \deg_j(b'_1) > \vphi_j(b'_2), \\
\vep_j(b'_2) \qu & \IF \vphi_j(b'_2) - \deg_j(b'_1) \in \Z_{\geq 0, \ev}, \\
\vep_j(b'_2)+1 \qu & \IF \vphi_j(b'_2) - \deg_j(b'_1) \in \Z_{\geq 0, \odd}
\end{cases} \\
&=\begin{cases}
\deg_j(b_1)-\vphi_j(b_2)+\vep_j(b_2)+a \qu & \IF \deg_j(b_1) > \vphi_j(b_2), \\
\vep_j(b_2)+a \qu & \IF \vphi_j(b_2) - \deg_j(b_1) \in \Z_{\geq 0, \ev}, \\
\vep_j(b_2)+a+1 \qu & \IF \vphi_j(b_2) - \deg_j(b_1) \in \Z_{\geq 0, \odd}
\end{cases} \\
&=\deg_j(b_1 \otimes b_2) + a \\
&\in \{ \deg_j(b_1 \otimes b_2) \pm 1 \}.
\end{split} \nonumber
\end{align}
Thus, our claim follows.

\item Let $b_1 \in \clB_1$, $b_2 \in \clB_2$, and $i,j \in [1,n-1]$ be such that $|i-j| > 1$ and $\Btil_i(b_1 \otimes b_2) \neq 0$. We show that $\deg_{j}(\Btil_{i}(b_1 \otimes b_2)) = \deg_j(b_1 \otimes b_2)$.

Let us write $\Btil_i(b_1 \otimes b_2) = b'_1 \otimes b'_2$ for some $b'_1 \in \{ b_1, \Btil_i b_1 \} \setminus \{0\}$ and $b'_2 \in \{ b_2, \Etil_i b_2, \Ftil_i b_2 \} \setminus \{0\}$. Then, we have
$$
\deg_{j}(b'_1) = \deg_j(b_1),\ \vep_j(b'_2) = \vep_j(b_2), \ \vphi_j(b'_2) = \vphi_j(b_2).
$$
This implies that
$$
\deg_j(b'_1 \otimes b'_2) = \deg_j(b_1 \otimes b_2),
$$
as desired.
\end{enumerate}
\end{proof}

\begin{cor}\label{AI-crystal structure on gl-crystal}
Let $\clB$ be a $\frgl$-crystal. Then, $\clB$ is equipped with an $\AI$-crystal structure as follows: For each $b \in \clB$ and $i \in [1,n-1]$, we set
\begin{align}
\begin{split}
&\deg_i(b) := \begin{cases}
\vep_i(b) \qu & \IF \vphi_i(b) \in \Z_{\ev}, \\
\vep_i(b)+1 \qu & \IF \vphi_i(b) \in \Z_{\odd},
\end{cases} \\
&\Btil_i b := \begin{cases}
\Etil_i b \qu & \IF \vphi_i(b) \in \Z_{\ev}, \\
\Ftil_i b \qu & \IF \vphi_i(b) \in \Z_{\odd}.
\end{cases}
\end{split} \nonumber
\end{align}
\end{cor}

\begin{proof}
Consider a $\frgl$-crystal $\SST_n(\emptyset) = \{ \emptyset \}$. Then, we have an isomorphism
$$
\SST_n(\emptyset) \otimes \clB \rightarrow \clB;\ \emptyset \otimes b \mapsto b
$$
of $\frgl$-crystals. On the other hand, $\SST_n(\emptyset)$ admits an $\AI$-crystal structure given by
$$
\Btil_i(\emptyset) = 0, \qu \deg_i(\emptyset) = 0.
$$
Then, under the identification $\SST_n(\emptyset) \otimes \clB \simeq \clB$, the $\AI$-crystal structure on $\SST_n(\emptyset) \otimes \clB$ given by Proposition \ref{AI-crystal tensor gl-crystal} is the same as the one on $\clB$ given by this corollary. Thus, the proof completes.
\end{proof}

\begin{rem}\normalfont
In the sequel, whenever we regard a $\frgl$-crystal as an $\AI$-crystal, we assume that its $\AI$-crystal structure is given by Corollary \ref{AI-crystal structure on gl-crystal}.
\end{rem}

\begin{ex}\normalfont
  The $\AI_2$-crystal graph of the $\mathfrak{gl}_3$-crystal $\SST_3(2,1)$, whose crystal graph is given in equation \eqref{eq: crystal graph of 2,1}, is as follows:
  \begin{align}\label{eq: AIcrystal graph of 2,1}
    \xymatrix@R=10pt{
      & \text{$\ytableausetup{centertableaux}
      \begin{ytableau}
        1 & 1 \\
        2
      \end{ytableau}$} \ar@{-}[dl]_1 \ar@{-}[dr]^2 & \\
      \text{$\ytableausetup{centertableaux}
      \begin{ytableau}
        1 & 2 \\
        2
      \end{ytableau}$} & & \text{$\ytableausetup{centertableaux}
      \begin{ytableau}
        1 & 1 \\
        3
      \end{ytableau}$} \\
      \text{$\ytableausetup{centertableaux}
      \begin{ytableau}
        1 & 3 \\
        2
      \end{ytableau}$} \ar@{-}[d]_2 & & \text{$\ytableausetup{centertableaux}
      \begin{ytableau}
        1 & 2 \\
        3
      \end{ytableau}$} \ar@{-}[d]^1 \\
      \text{$\ytableausetup{centertableaux}
      \begin{ytableau}
        1 & 3 \\
        3
      \end{ytableau}$} \ar@{-}[dr]_1 & & \text{$\ytableausetup{centertableaux}
      \begin{ytableau}
        2 & 2 \\
        3
      \end{ytableau}$} \ar@{-}[dl]^2 \\
      & \text{$\ytableausetup{centertableaux}
      \begin{ytableau}
        2 & 3 \\
        3
      \end{ytableau}$}. &
    }
  \end{align}
\end{ex}

\begin{prop}\label{degree 0 in Stembridge crystals}
Let $\clB$ be a $\frgl$-crystal, $b \in \clB$, and $i,j \in [1,n-1]$.
\begin{enumerate}
\item\label{degree 0 in Stembridge crystals 1} We have $\Btil_i b = 0$ if and only if $\deg_i(b) = 0$.
\item\label{degree 0 in Stembridge crystals 2} If $|i-j| > 1$ and $\Btil_i b \neq 0$, then we have $\Btil_j \Btil_i b = \Btil_i \Btil_j b$.
\end{enumerate}
\end{prop}

\begin{proof}
Let us prove the first assertion. Suppose that $\Btil_i b = 0$. We show that $\vphi_i(b) \in \Z_{\ev}$. Otherwise, we have
$$
0 = \Btil_i b = \Ftil_i b,
$$
which implies $\vphi_i(b) = 0$. This is a contradiction. Hence, we obtain $\vphi_i(b) \in \Z_{\ev}$, and consequently,
$$
0 = \Btil_i b = \Etil_i b.
$$
This shows that $\vep_i(b) = 0$, and hence, we have
$$
\deg_i(b) = \vep_i(b) = 0.
$$

Conversely, suppose that $\deg_i(b) = 0$. Then, we must have $\vphi_i(b) \in \Z_{\ev}$; otherwise, $\deg_i(b) = \vep_i(b) + 1 > 0$. Hence, we obtain
$$
0 = \deg_i(b) = \vep_i(b),
$$
and consequently,
$$
\Btil_i b = \Etil_i b = 0,
$$
as desired. This completes the proof of the first assertion.

Now, let us prove the second assertion. Note that for each $X,Y \in \{E,F\}$, we have
$$
\vphi_j(\Xtil_i b) = \vphi_j(b), \qu \vphi_i(\Ytil_j b) = \vphi_i(b),
$$
and
$$
\Ytil_j \Xtil_i b = \Xtil_i \Ytil_j b.
$$
Then, if we write $\Btil_i b = \Xtil_i b$ for some $X \in \{E,F\}$, we compute as
\begin{align}
\begin{split}
\Btil_j \Btil_i b &= \Btil_j \Xtil_i b \\
&= \begin{cases}
\Etil_j \Xtil_i b \qu & \IF \vphi_j(b) \in \Z_{\ev}, \\
\Ftil_j \Xtil_i b \qu & \IF \vphi_j(b) \in \Z_{\odd}
\end{cases} \\
&= \begin{cases}
\Xtil_i \Etil_j b \qu & \IF \vphi_j(b) \in \Z_{\ev}, \\
\Xtil_i \Ftil_j b \qu & \IF \vphi_j(b) \in \Z_{\odd}
\end{cases} \\
&= \Xtil_i \Btil_j b \\
&= \begin{cases}
\Etil_i \Btil_j b \qu & \IF \vphi_i(b) \in \Z_{\ev}, \\
\Ftil_i \Btil_j b \qu & \IF \vphi_i(b) \in \Z_{\odd}
\end{cases} \\
&= \Btil_i \Btil_j b.
\end{split} \nonumber
\end{align}
This implies the assertion.
\end{proof}

\begin{prop}
Let $\clB_1$ be an $\AI$-crystal, and $\clB_2,\clB_3$ be $\frgl$-crystals. Then, we have
$$
(\clB_1 \otimes \clB_2) \otimes \clB_3 = \clB_1 \otimes (\clB_2 \otimes \clB_3).
$$
\end{prop}

\begin{proof}
The proof is straightforward but long. Hence, we omit it.
\end{proof}

\begin{rem}\normalfont
Let $\clB_1,\clB_2$ be $\frgl$-crystals. Then, we can equip $\clB_1 \otimes \clB_2$ with two $\AI$-crystal structures; one is obtained by regarding $\clB_1$ as an $\AI$-crystal by means of Corollary \ref{AI-crystal structure on gl-crystal} and then by taking tensor product, and the other is obtained by regarding the $\frgl$-crystal $\clB_1 \otimes \clB_2$ as an $\AI$-crystal by means of Corollary \ref{AI-crystal structure on gl-crystal}. These two structures are identical because the former is $(\SST_n(\emptyset) \otimes \clB_1) \otimes \clB_2$, while the latter is $\SST_n(\emptyset) \otimes (\clB_1 \otimes \clB_2)$.
\end{rem}

\begin{cor}\label{AI-crystal morphism otimes 1}
Let $\clB_1,\clB_2$ be $\AI$-crystals, $\clB_3,\clB_4$ be $\frgl$-crystals, $\psi_1 : \clB_1 \rightarrow \clB_2$ a morphism of $\AI$-crystals, and $\psi_2 : \clB_3 \rightarrow \clB_4$ a morphism of $\frgl$-crystals. Then,
$$
\psi_1 \otimes \psi_2 : \clB_1 \otimes \clB_3 \rightarrow \clB_2 \otimes \clB_4;\ b_1 \otimes b_2 \mapsto \psi_1(b_1) \otimes \psi_2(b_2)
$$
is a morphism of $\AI$-crystals. Furthermore, if both $\psi_1$ and $\psi_2$ are isomorphisms, then so is $\psi_1 \otimes \psi_2$.
\end{cor}

\begin{proof}
The assertion follows from Proposition \ref{AI-crystal tensor gl-crystal}. In fact, $\deg_i(b_1 \otimes b_2)$ (resp., $\deg_i(\psi_1(b_1) \otimes \psi_2(b_2))$) is determined by $\deg_i(b_1),\vep_i(b_2),\vphi_i(b_2)$ (resp., $\deg_i(\psi(b_1)),\vep_i(\psi(b_2)),\vphi_i(\psi(b_2))$), and $\Btil_i$ acts on $b_1 \otimes b_2$ (resp., $\psi_1(b_1) \otimes \psi_2(b_2)$) by either $\Btil_i \otimes 1$, $1 \otimes \Etil_i$, or $1 \otimes \Ftil_i$ depending on $\deg_i(b_1),\vep_i(b_2),\vphi_i(b_2)$ (resp., $\deg_i(\psi(b_1)),\vep_i(\psi(b_2)),\vphi_i(\psi(b_2))$).
\end{proof}


\begin{defi}\normalfont
Let $\clB$ be an $\AI$-crystal. Define its character $\ch_{\AI} \clB$ to be a Laurent polynomial in $m$ variables $y_1,y_3,\ldots,y_{2m-1}$ given by
$$
\frac{1}{2^m} \sum_{b \in \clB} \sum_{\sigma_1,\sigma_3,\ldots,\sigma_{2m-1} \in \{ +,- \}} y_1^{\sigma_1 \deg_1(b)} y_3^{\sigma_3 \deg_3(b)} \cdots y_{2m-1}^{\sigma_{2m-1} \deg_{2m-1}(b)} \in \Z[y_1^{\pm 1},y_3^{\pm 1}, \ldots, y_{2m-1}^{\pm 1}].
$$
\end{defi}

\begin{ex}\normalfont
We have
$$
\ch_{\AI_2} \SST_3(1) = y_1 + 1 + y_1\inv, \qu \ch_{\AI_{3}} \SST_4(1) = y_1 + y_1\inv + y_3 + y_3\inv.
$$
\end{ex}

Let us explain the meaning of the character of an $\AI$-crystal $\clB$. Assume that $\clB$ has the following property:
\begin{enumerate}
\item For each $b \in \clB$ and $i \in [1,m]$, we have $\Btil_{2i-1}(b) = 0$ if and only if $\deg_{2i-1}(b) = 0$.
\item For each $b \in \clB$ and $i \neq j \in [1,m]$, we have $\Btil_{2i-1} \Btil_{2j-1} b = \Btil_{2j-1} \Btil_{2i-1} b$.
\end{enumerate}
For example, an $\AI$-subcrystal of a $\frgl$-crystal admits this property (see Proposition \ref{degree 0 in Stembridge crystals}). Set $\ol{\clL} := \C \clB$, and extend the maps $\Btil_i$ to linear operators on $\ol{\clL}$. Note that the operators $\Btil_1,\Btil_3,\ldots,\Btil_{2m-1}$ pairwise commute. We say that a vector $u \in \ol{\clL}$ is a weight vector of weight $\nu = (\nu_1,\nu_3,\ldots,\nu_{2m-1}) \in \Z^m$ if it satisfies the following:
\begin{enumerate}
\item $u$ is a linear combination of $b \in \clB$ such that $\deg_{2i-1}(b) = |\nu_{2i-1}|$ for all $i \in [1,m]$.
\item $\Btil_{2i-1}(u) = \begin{cases}
u \qu & \IF \nu_{2i-1} > 0, \\
0 \qu & \IF \nu_{2i-1} = 0, \\
-u \qu & \IF \nu_{2i-1} < 0, \\
\end{cases}$
\end{enumerate}
Let $\ol{\clL}_\nu$ denote the subspace of weight vectors of weight $\nu$. Then, $\ol{\clL}$ admits the weight space decomposition
$$
\ol{\clL} = \bigoplus_{\nu \in \Z^m} \ol{\clL}_\nu.
$$
In fact, if we take a complete set $\clB'$ of representatives for $\clB/{\sim}$ with respect to the equivalence relation given by
$$
b_1 \sim b_2 \text{ if and only if } b_2 \in \{ \Btil_{2i_1-1} \Btil_{2i_2-1} \cdots \Btil_{2i_r-1} b_1 \mid r \in [0,m],\  1 \leq i_1 < \cdots < i_r \leq m \},
$$
then the set
$$
\{ (1 + \sigma_1 \Btil_1)(1 + \sigma_3 \Btil_3) \cdots (1 + \sigma_{2m-1} \Btil_{2m-1})b \mid b \in \clB',\ \deg_{2i-1}(b) = |\nu_{2i-1}| \Forall i \in [1,m] \},
$$
where $\sigma_{2i-1} \in \{+,-\}$ denotes the signature of $\nu_{2i-1}$, forms a basis of $\ol{\clL}_\nu$. Therefore, we have
$$
\ch_{\AI} \clB = \sum_{\nu \in \Z^m} (\dim \ol{\clL}_\nu) \bfy^\nu,
$$
where
$$
\bfy^{(\nu_1,\nu_3,\ldots,\nu_{2m-1})} := y_1^{\nu_1} y_3^{\nu_3} \cdots y_{2m-1}^{\nu_{2m-1}}.
$$

\subsection{K-matrices}
In this subsection, we introduce a family of isomorphisms of $\AI$-crystals. They are closely related to $K$-matrices appearing in the representation theory of $\imath$quantum group of type AI.

Let $k \in [0,n]$ and consider a $\frgl$-crystal $\SST_n(1^k)$, where
$$
1^k := \begin{cases}
(\overbrace{1,\ldots,1}^{k}) \qu & \IF k \neq 0, \\
\emptyset \qu & \IF k = 0.
\end{cases}
$$
For each $1 \leq j_1 < \cdots < j_k \leq n$, set
$$
u_{j_1,\ldots,j_k} := \begin{cases}
\ytableausetup{centertableaux}
\begin{ytableau}
j_1 \\
\vdots \\
j_k
\end{ytableau} \qu & \IF k \neq 0, \\
\\
\emptyset \qu & \IF k = 0.
\end{cases}
$$
Then, we have
$$
\SST_n(1^k) = \begin{cases}
\{ u_{j_1,\ldots,j_k} \mid 1 \leq j_1 < \cdots < j_k \leq n \} \qu & \IF k \neq 0, \\
\{ \emptyset \} \qu & \IF k = 0.
\end{cases}
$$
The $\AI$-crystal structure of $\SST_n(1^k)$ can be easily described as follows.

\begin{lem}\label{lemma for K-matrix}
Let $1 \leq j_1 < \cdots < j_k \leq n$ and $i \in [1,n-1]$. Then, we have
\begin{align}
\begin{split}
&\deg_i(u_{j_1,\ldots,j_k}) = \begin{cases}
1 & \IF |\{ j_1,\ldots,j_k \} \cap \{ i,i+1 \}| = 1 \\
0 & \OW,
\end{cases} \\
&\Btil_i u_{j_1,\ldots,j_k} = \begin{cases}
u_{j_1,\ldots,j_{l-1},i,j_{l+1},\ldots,j_k} & \IF j_{l-1} < i = j_l-1 \Forsome l \in [1,k], \\
u_{j_1,\ldots,j_{l-1},i+1,j_{l+1},\ldots,j_k} & \IF j_l+1 = i+1 <  j_{l+1} \Forsome l \in [1,k], \\
0 \qu & \OW.
\end{cases}
\end{split} \nonumber
\end{align}
\end{lem}

\begin{defi}\normalfont
For each $k \in [0,n]$, define a map $K = K^{(k)} : \SST_n(1^k) \rightarrow \SST_n(1^{n-k})$ by
$$
K(u_{j_1,\ldots,j_k}) = u_{j^c_1,\ldots,j^c_{n-k}},
$$
where $\{ j^c_1,\ldots,j^c_{n-k} \} = [1,n] \setminus \{ j_1,\ldots,j_k \}$. 
\end{defi}

\begin{prop}\label{K-matrix is isomorphism}
Let $k \in [0,n]$. Then, $K^{(k)} : \SST_n(1^k) \rightarrow \SST_n(1^{n-k})$ is an isomorphism of $\AI$-crystals with inverse $K^{(n-k)} : \SST_n(1^{n-k}) \rightarrow \SST_n(1^k)$.
\end{prop}

\begin{proof}
It is clear from the definition that $K^{(k)}$ is a bijection with inverse $K^{(n-k)}$. For each $1 \leq j_1 < \cdots < j_k \leq n$ and $i \in [1,n-1]$, by Lemma \ref{lemma for K-matrix}, we see that
$$
K(\Btil_i u_{j_1,\ldots,j_k}) = \Btil_i K(u_{j_1,\ldots,j_k}),
$$
and
$$
\deg_i(K(u_{j_1,\ldots,j_k})) = \deg_i(u_{j_1,\ldots,j_k}).
$$
Thus, the proof completes.
\end{proof}

\begin{cor}\label{K otimes 1 is AI-isomorphism}
Let $\clB$ be a $\frgl$-crystal. Then, for each $k \in [0,n]$, the map
$$
K \otimes 1 : \SST_n(1^k) \otimes \clB \rightarrow \SST_n(1^{n-k}) \otimes \clB;\ b_1 \otimes b_2 \mapsto K(b_1) \otimes b_2
$$
is an isomorphism of $\AI$-crystals.
\end{cor}

\begin{proof}
The assertion follows from Proposition \ref{K-matrix is isomorphism} and Corollary \ref{AI-crystal morphism otimes 1}.
\end{proof}

Let $\lm \in \Par_{n}$. For each $j \in [1,\lm_1]$, let $d_j$ denote the length of the $j$-th column of $D(\lm)$. By the definition of the $\frgl$-crystal structure of $\SST_n(\lm)$, there exists an embedding
$$
\SST_n(\lm) \hookrightarrow \SST_n(1^{d_1}) \otimes \SST_n(1^{d_2}) \otimes \cdots \otimes \SST_n(1^{d_{\lm_1}})
$$
of a $\frgl$-crystal which sends $T = C_1 \cdots C_{\lm_1}$ to $C_1 \otimes \cdots \otimes C_{\lm_1}$, where $C_j$ denotes the $j$-th column of $T$. Define a new semistandard Young tableau $K_1(T)$ by
$$
K_1(T) := P(K(C_1) \otimes C_2 \otimes \cdots \otimes C_{\lm_1}).
$$
Then, the following are immediate from the definition of $K_1$ and Corollary \ref{K otimes 1 is AI-isomorphism}.

\begin{prop}\label{Size of K1(T)}
Let $\lm \in \Par_n$ and $T \in \SST_n(\lm)$. Then, we have
$$
|K_1(T)| - |T| = n-2\ell(\lm).
$$
\end{prop}

\begin{prop}\label{K1 is an AI-crystal morphism}
Let $\lm \in \Par_{n}$ and $T \in \SST_n(\lm)$. Then, for each $i \in [1,n-1]$, we have
$$
\Btil_i(K_1(T)) = K_1(\Btil_i T), \qu \deg_i(K_1(T)) = \deg_i(T).
$$
Here, we set $K_1(0) := 0$.
\end{prop}

\subsection{$\AI$-tableaux}
In this subsection, we introduce $\AI$-tableaux, which are central objects in this paper. They provide us many concrete examples of $\AI$-crystals which are not $\frgl$-crystals.

\begin{defi}\normalfont\label{Definition of AI-semistandard}
Let $\lm \in \Par_{n}$ and $T \in \SST_n(\lm)$. We say that $T$ satisfies the $\AI$-condition, or $T$ is an $\AI$-tableau, if it satisfies the following two conditions:
\begin{enumerate}
\item\label{Definition of AI-semistandard 1} $d_1 \leq m$ (see equation \eqref{eq: def of m} for the definition of $m$).
\item\label{Definition of AI-semistandard 2} $t^c_{i,1} \leq t_{i,2}$ for all $i \in [1,d_2]$, where $\{ t^c_{1,1},\ldots,t^c_{n-d_1,1} \} = [1,n] \setminus \{ t_{1,1},\ldots,t_{d_1,1} \}$.
\end{enumerate}
Here, $d_j$ denotes the length of the $j$-th column of $D(\lm)$, and $t_{i,j}$ denotes the $(i,j)$-th entry of $T$. For each $\lm \in \Par_{n}$, let $\SST_n^{\AI}(\lm)$ denote the set of semistandard Young tableaux of shape $\lm$ satisfying the $\AI$-condition.
\end{defi}

\begin{rem}\label{AI-condition in other words}\normalfont
By Lemma \ref{observation of C1C2}, the second condition for $\AI$-tableaux is equivalent to saying that if we write $T = C_1 C_2 \cdots C_{\lm_1}$, where $C_j$ denotes the $j$-th column of $T$, then
$$
K_1(T) = K(C_1) C_2 \cdots C_{\lm_1}.
$$
\end{rem}

\begin{rem}\normalfont
If $T$ satisfies the $\AI$-condition for some $n$, then so does for all $n' \geq n$. However, $\AI$-condition depends on $n$, in general. For example, $\ytableausetup{smalltableaux}
\begin{ytableau}
1 \\
2
\end{ytableau}$ is an $\AI$-tableau when $n \geq 4$, but not when $n = 3$.
\end{rem}

\begin{rem}\normalfont
It is clear that $\SST_n^{\AI}(\lm) = \emptyset$ unless $\ell(\lm) \leq m$.
\end{rem}

\begin{ex}\normalfont
\ \begin{enumerate}
\item Let $n  \geq 3$, $\lm = (l)$, $l \geq 0$. Then, $T \in \SST_n(l)$ is an $\AI$-tableau if and only if $T$ does not begin with $\ytableausetup{centertableaux}
\begin{ytableau}
1 & 1
\end{ytableau}$. For example,
$$
\SST_3^{\AI}(2) = \{ \ytableausetup{centertableaux}
\begin{ytableau}
1 & 2
\end{ytableau},\ \ytableausetup{centertableaux}
\begin{ytableau}
1 & 3
\end{ytableau},\ \ytableausetup{centertableaux}
\begin{ytableau}
2 & 2
\end{ytableau},\ \ytableausetup{centertableaux}
\begin{ytableau}
2 & 3
\end{ytableau},\ \ytableausetup{centertableaux}
\begin{ytableau}
3 & 3
\end{ytableau} \}
$$
\item Let $n \geq 4$, $\lm = (l_1,l_2)$, $l_1 \geq l_2 > 0$. Then, $T \in \SST_n(l_1,l_2)$ is an $\AI$-tableau if and only if the following conditions are satisfied:
\begin{enumerate}
\item The first row of $T$ does not begin with $\ytableausetup{centertableaux}
\begin{ytableau}
1 & 1
\end{ytableau}$.
\item The second row of $T$ does not begin with $\ytableausetup{centertableaux}
\begin{ytableau}
3 & 3
\end{ytableau}$.
\item The first two rows of $T$ does not begin with $
\ytableausetup{centertableaux}
\begin{ytableau}
1 & 2 \\
2
\end{ytableau}$.
\end{enumerate}
For example,
\begin{align}
\begin{split}
\SST_4^{\AI}(2,1) = \{ &\ytableausetup{centertableaux}
\begin{ytableau}
1 & 2 \\
3
\end{ytableau}, \qu \ytableausetup{centertableaux}
\begin{ytableau}
1 & 2 \\
4
\end{ytableau}, \qu \ytableausetup{centertableaux}
\begin{ytableau}
1 & 3 \\
2
\end{ytableau}, \qu \ytableausetup{centertableaux}
\begin{ytableau}
1 & 3 \\
3
\end{ytableau}, \qu \ytableausetup{centertableaux}
\begin{ytableau}
1 & 3 \\
4
\end{ytableau}, \qu \ytableausetup{centertableaux}
\begin{ytableau}
1 & 4 \\
2
\end{ytableau}, \qu \ytableausetup{centertableaux}
\begin{ytableau}
1 & 4 \\
3
\end{ytableau}, \qu \ytableausetup{centertableaux}
\begin{ytableau}
1 & 4 \\
4
\end{ytableau}, \\
&\ytableausetup{centertableaux}
\begin{ytableau}
2 & 2 \\
3
\end{ytableau}, \qu \ytableausetup{centertableaux}
\begin{ytableau}
2 & 2 \\
4
\end{ytableau}, \qu \ytableausetup{centertableaux}
\begin{ytableau}
2 & 3 \\
3
\end{ytableau}, \qu \ytableausetup{centertableaux}
\begin{ytableau}
2 & 3 \\
4
\end{ytableau}, \qu \ytableausetup{centertableaux}
\begin{ytableau}
2 & 4 \\
3
\end{ytableau}, \qu \ytableausetup{centertableaux}
\begin{ytableau}
2 & 4 \\
4
\end{ytableau}, \qu \ytableausetup{centertableaux}
\begin{ytableau}
3 & 3 \\
4
\end{ytableau}, \qu \ytableausetup{centertableaux}
\begin{ytableau}
3 & 4 \\
4
\end{ytableau} \}.
\end{split} \nonumber
\end{align}
\end{enumerate}
\end{ex}

%
%

\begin{lem}\label{Existence of AI-semistandardization}
Let $\lm \in \Par_{n}$ and $T \in \SST_n(\lm)$. Then, there exists $r \geq 0$ such that $K_1^r(T)$ is an $\AI$-tableau.
\end{lem}

\begin{proof}
We prove by induction on $|T|$. When $T$ is an $\AI$-tableau, there is nothing to prove. Suppose that $T$ is not an $\AI$-tableau. Set $d_1 := \ell(\lm)$. Then, we have two possibilities; (1) $d_1 > m$ or (2) $d_1 \leq m$ and the second condition in Definition \ref{Definition of AI-semistandard} fails. If we are in the first case, then by Proposition \ref{Size of K1(T)}, we have
$$
|K_1(T)| - |T| = n-2d_1 \leq n-2(m+1) < 0.
$$
Hence, the assertion follows from our induction hypothesis.

Now, assume that we are in the second case. Then, by Remark \ref{AI-condition in other words} and Lemma \ref{observation of C1C2}, the length of $\sh(K_1(T))$ must be $n-d_1+\alpha$ for some $\alpha > 0$. Using Proposition \ref{Size of K1(T)}, we compute as
\begin{align}
\begin{split}
|K_1^2(T)| - |T| &= (|K_1^2(T)|-|K_1(T)|)+ (|K_1(T)|-|T|) \\
&= (n-2(n-d_1+\alpha)) + (n-2d_1) \\
&= -2\alpha < 0.
\end{split} \nonumber
\end{align}
Hence, the assertion follows from our induction hypothesis. This completes the proof.
\end{proof}

\begin{lem}\label{K1s on AI-semistandard tableau}
Let $\rho \in \Par_m$, $T \in \SST^{\AI}_n(\rho)$. Then, we have $K_1^2(T) = T$. Furthermore, $K_1(T)$ is an $\AI$-tableau if and only if $\ell(\rho) = \frac{n}{2}$.
\end{lem}

\begin{proof}
For each $j \in [1,\lm_1]$, let $d_j$ denote the length of the $j$-th column of $D(\lm)$, and $C_j$ the $j$-th column of $T$. By Remark \ref{AI-condition in other words}, we have
\begin{align}\label{K1 on AI-tableau}
K_1(T) = K(C_1)C_2 \cdots C_{\lm_1}.
\end{align}
Since $T$ is a semistandard Young tableau, by Lemma \ref{observation of C1C2}, we see that
\begin{align}\label{K1 squared T}
K_1^2(T) = K^2(C_1)C_2 \cdots C_{\lm_1} = C_1C_2 \cdots C_{\lm_1} = T.
\end{align}
This implies the first assertion. Combining Remark \ref{AI-condition in other words} and equation \eqref{K1 squared T}, we see that $K_1(T)$ is an $\AI$-tableau if and only if $\ell(\sh(K_1(T))) \leq m$. Since $\ell(\sh(K_1(T))) = n-\ell(\lm)$ by equation \eqref{K1 on AI-tableau}, and since $\ell(\lm) \leq m$, the second assertion follows. This completes the proof.
\end{proof}

\begin{defi}\normalfont
Let $\lm \in \Par_{n}$ and $T \in \SST_n(\lm)$. The $P^{\AI}$-symbol of $T$ is an $\AI$-tableau $\std(T)$ given by
$$
\std(T) := K_1^r(T),
$$
where $r := \min\{ s \mid \text{$K_1^s(T)$ is an $\AI$-tableau} \}$.
\end{defi}

\begin{ex}\normalfont
Let $n = 4$ and $T = \ytableausetup{centertableaux}
\begin{ytableau}
2 & 2 \\
3 & 3
\end{ytableau} \notin \SST_4^{\AI}(2,2)$. Then, we compute as
\begin{align}
\begin{split}
K_1 \left( \ytableausetup{centertableaux}
\begin{ytableau}
2 & 2 \\
3 & 3
\end{ytableau} \right) &= P \left( \ytableausetup{centertableaux}
\begin{ytableau}
1 \\
4
\end{ytableau} \otimes \ytableausetup{centertableaux}
\begin{ytableau}
2 \\
3
\end{ytableau} \right) = \ytableausetup{centertableaux}
\begin{ytableau}
1 & 2 \\
3 \\
4
\end{ytableau} \notin \SST_4^{\AI}(2,1,1), \\
K_1 \left( \ytableausetup{centertableaux}
\begin{ytableau}
1 & 2 \\
3 \\
4
\end{ytableau} \right) &= P \left( \Cbox{2} \otimes \Cbox{2} \right) = \ytableausetup{centertableaux}
\begin{ytableau}
2 & 2
\end{ytableau} \in \SST_4^{\AI}(2).
\end{split} \nonumber
\end{align}
Hence, we obtain $\std(T) = \ytableausetup{centertableaux}
\begin{ytableau}
2 & 2
\end{ytableau}$.
\end{ex}

\begin{rem}\normalfont
$\std(T)$ depends on $n$, in general. For example, we have
$$
\std \left( \ytableausetup{centertableaux}
\begin{ytableau}
1 \\
2
\end{ytableau} \right) = \begin{cases}
\Cbox{3} \qu & \IF n = 3, \\
\\
\ytableausetup{centertableaux}
\begin{ytableau}
1 \\
2
\end{ytableau} \qu & \IF n > 3.
\end{cases}
$$
\end{rem}

\begin{prop}\label{SSTnAI is AI-crystal}
Let $\rho \in \Par_m$. Then, $\SST_n^{\AI}(\rho)$ is an $\AI$-subcrystal of $\SST_n(\rho)$.
\end{prop}

\begin{proof}
Let $T \in \SST_n^{\AI}(\rho)$ and $i \in [1,n-1]$ be such that $\Btil_i T \neq 0$. It suffices to show that $\Btil_i T \in \SST_n^{\AI}(\rho)$. Assume contrary. Then, by the proof of Lemma \ref{Existence of AI-semistandardization}, we have
$$
|\std(\Btil_i T)| < |\Btil_i T| = |T|.
$$
On the other hand, if we write $\std(\Btil_i T) = K_1^r(\Btil_i T)$ for some $r > 0$, then we have
$$
\std(\Btil_i T) = \Btil_i K_1^r(T)
$$
by Proposition \ref{K1 is an AI-crystal morphism}. Hence, we compute as
$$
|\std(\Btil_i T)| = |\Btil_i K_1^r(T)| = |K_1^r(T)| \geq |T|,
$$
where the last inequality follows from Lemma \ref{K1s on AI-semistandard tableau}. Thus, we obtain a contradiction. Hence, the proof completes.
\end{proof}

\begin{prop}\label{PAI-symbol is AI-crystal morphism}
The map
$$
\std : \bigsqcup_{\lm \in \Par_{n}} \SST_n(\lm) \rightarrow \bigsqcup_{\rho \in \Par_{m}}\SST_n^{\AI}(\rho);\ T \mapsto \std(T)
$$
is a morphism of $\AI$-crystals.
\end{prop}

\begin{proof}
Let $\lm \in \Par_{n}$, $T \in \SST_n(\lm)$, $i \in [1,n-1]$. Let $r \geq 0$ be the minimum integer such that $\std(T) = K_1^r(T)$. By Proposition \ref{K1 is an AI-crystal morphism}, we see that
$$
\Btil_i \std(T) = K_1^r(\Btil_i T), \qu \deg_i(\std(T)) = \deg_i(T).
$$
Hence, it suffices to show that
$$
\std(\Btil_i T) = K_1^r(\Btil_i T).
$$
First, suppose that $\Btil_i T = 0$. Since $K_1^r$ is a morphism of $\AI$-crystals, our claim follows immediately.

Next, suppose that $\Btil_i T \neq 0$. Let $r' \geq 0$ be the minimum integer such that
$$
\std(\Btil_i T) = K_1^{r'}(\Btil_i T).
$$
By Proposition \ref{SSTnAI is AI-crystal}, we see that $\Btil_i \std(T)$, which equals $K_1^r(\Btil_i T)$, is an $\AI$-tableau. By the minimality of $r'$, we have $r' \leq r$. Let us show that $r' = r$. If $r' < r$, then
$$
K_1^{r'}(T) = \Btil_i^2 K_1^{r'}(T) = \Btil_i K_1^{r'}(\Btil_i T) = \Btil_i \std(\Btil_i T).
$$
Here, the first equality holds because $\Btil_i T \neq 0$, and hence,
$$
\Btil_i K_1^{r'}(T) = K_1^{r'}(\Btil_i T) \neq 0.
$$
Since $\std(\Btil_i T)$ is an $\AI$-tableau, so is $\Btil_i \std(\Btil_i T)$, which equals $K_1^{r'}(T)$. This contradicts the minimality of $r$. Thus, we obtain $r' = r$, and hence,
$$
\std(\Btil_i T) = K_1^{r}(\Btil_i T),
$$
as desired.
\end{proof}

\subsection{Low rank examples}\label{Subsection, low rank examples}
In this subsection, we investigate the $\AI$-crystal structures of $\SST_n^{\AI}(\rho)$, $\rho \in \Par_m$ in the case when $n = 3,4$.

First, assume that $n = 3$ and consider $\SST_3^{\AI}(l)$, $l \geq 0$. For each $a \in \{ 1,2 \}$ and $b \in [0,l-1]$, set
$$
T_{a,b} := \ytableausetup{centertableaux}
\begin{ytableau}
a & 2 & 2 & \cdots & 2 & 3 & 3 & \cdots & 3
\end{ytableau} \in \SST_3^{\AI}(l),
$$
where $b$ is the number of $3$'s. Also, set
$$
T_l := \ytableausetup{centertableaux}
\begin{ytableau}
3 & 3 & \cdots & 3
\end{ytableau} \in \SST_3^{\AI}(l).
$$
Then, we have
$$
\SST_3^{\AI}(l) = \{ T_{a,b} \mid a \in \{1,2\},\ b \in [0,l-1] \} \sqcup \{ T_l \},
$$
and hence,
\begin{align}\label{size of SST3AI(l)}
|\SST_3^{\AI}(l)| = 2l+1.
\end{align}

\begin{lem}\label{Connectedness of SST3AI(l)}
Let $n = 3$ and $l \geq 0$. Then, the $\AI_2$-crystal $\SST_3^{\AI}(l)$ is connected.
\end{lem}

\begin{proof}
From definitions, we obtain
\begin{align}
\begin{split}
&\Btil_1 T_{a,b} = T_{a',b}, \\
&\Btil_2 T_{a,b} = \begin{cases}
0 \qu & \IF l-\delta_{a,1}-b \in \Z_{\ev} \AND b = 0, \\
T_{a,b-1} \qu & \IF l-\delta_{a,1}-b \in \Z_{\ev} \AND b > 0, \\
T_{a,b+1} \qu & \IF l-\delta_{a,1}-b \in \Z_{\odd} \AND b < l-1, \\
T_l \qu & \IF l-\delta_{a,1}-b \in \Z_{\odd} \AND  b = l-1,
\end{cases} \\
&\Btil_1 T_l = 0, \\
&\Btil_2 T_l = T_{2,l-1},
\end{split} \nonumber
\end{align}
where $a' \in \{ 1,2 \} \setminus \{a\}$. Then, we see that
$$
\Btil_2 \Btil_1 T_{1,l-1} = \Btil_2 T_{2,l-1} = T_l,
$$
and
$$
\Btil_2 \Btil_1 T_{a',b} = \Btil_2 T_{a,b} = T_{a,b+1}
$$
for all $b \in [0,l-2]$, where $a = 1$ if $l-b \in \Z_{\ev}$ and $a = 2$ otherwise, and $a' \in \{ 1,2 \} \setminus \{a\}$. These show that each $T \in \SST_3^{\AI}(\rho)$ is connected to $T_l$. Therefore, $\SST_3^{\AI}(\rho)$ is connected.
\end{proof}

Next, assume that $n = 4$ and consider $\SST_4^{\AI}(l)$, $l \geq 0$. For each $a \in \{ 1,2 \}$, $c \in [0,l-1]$, and $b \in [0,l-c-1]$, set
$$
T_{a,b,c} := \ytableausetup{centertableaux}
\begin{ytableau}
a & 2 & 2 & \cdots & 2 & 3 & 3 & \cdots & 3 & 4 & 4 & \cdots & 4
\end{ytableau} \in \SST_4^{\AI}(l),
$$
where $b$ and $c$ are the numbers of $3$'s and $4$'s, respectively. Also, for each $c \in [0,l]$, set
$$
T_c := \ytableausetup{centertableaux}
\begin{ytableau}
3 & 3 & \cdots & 3 & 4 & 4 & \cdots & 4
\end{ytableau} \in \SST_4^{\AI}(l),
$$
where $c$ is the number of $4$'s. Then, we have
$$
\SST_4^{\AI}(l) = \{ T_{a,b,c} \mid a \in \{1,2\},\ c \in [0,l-1],\ b \in [0,l-c-1] \} \sqcup \{ T_c \mid c \in [0,l] \},
$$
and hence,
\begin{align}\label{size of SST4AI(l)}
|\SST_4^{\AI}(l)| = \sum_{c=0}^{l-1} 2(l-c) + (l+1) = (l+1)^2.
\end{align}

\begin{lem}\label{Connectedness of SST4AI(l)}
Let $n = 4$ and $l \geq 0$. Then, the $\AI_3$-crystal $\SST_4^{\AI}(l)$ is connected.
\end{lem}

\begin{proof}
From definitions, we obtain
\begin{align}
\begin{split}
&\Btil_1 T_{a,b,c} = T_{a',b,c}, \\
&\Btil_2 T_{a,b,c} = \begin{cases}
0 \qu & \IF l-\delta_{a,1}-b-c \in \Z_{\ev} \AND b = 0, \\
T_{a,b-1,c} \qu & \IF l-\delta_{a,1}-b-c \in \Z_{\ev} \AND b > 0, \\
T_{a,b+1,c} \qu & \IF l-\delta_{a,1}-b-c \in \Z_{\odd} \AND b < l-c-1, \\
T_c \qu & \IF l-\delta_{a,1}-b-c \in \Z_{\odd} \AND b = l-c-1,
\end{cases} \\
&\Btil_3 T_{a,b,c} = \begin{cases}
0 \qu & \IF b \in \Z_{\ev} \AND c = 0, \\
T_{a,b+1,c-1} \qu & \IF b \in \Z_{\ev} \AND c > 0, \\
T_{a,b-1,c+1} \qu & \IF b \in \Z_{\odd}, 
\end{cases} \\
&\Btil_1 T_c = 0, \\
&\Btil_2 T_c = \begin{cases}
0 \qu & \IF l-c = 0, \\
T_{2,l-c-1,c} \qu & \IF l-c > 0,
\end{cases} \\
&\Btil_3 T_c = \begin{cases}
0 \qu & \IF l-c \in \Z_{\ev} \AND c = 0, \\
T_{c-1} \qu & \IF l-c \in \Z_{\ev} \AND c > 0, \\
T_{c+1} \qu & \IF l-c \in \Z_{\odd}.
\end{cases}
\end{split} \nonumber
\end{align}
where $a' \in \{ 1,2 \} \setminus \{a\}$. We show that each $T_{a,b,c}$ and $T_c$ are connected to $T_{1,0,0}$ by induction on $c$. When $c = 0$, our claim follows from Lemma \ref{Connectedness of SST3AI(l)}. Assume that $c > 0$. Then, by Lemma \ref{Connectedness of SST3AI(l)}, we see that $T_{a,b,c}$ is connected to $T_c$. Then, we have $\Btil_3(T_c) = T_{c-1}$ (resp., $\Btil_3 \Btil_2 T_c = T_{2,l-c,c-1}$) if $l-c$ is even (resp., $l-c$ is odd). Hence, our induction hypothesis implies that $T_c$ is connected to $T_{1,0,0}$. Thus, the proof completes.
\end{proof}

Finally, assume that $n = 4$ and consider $\SST_4^{\AI}(l_1,l_2)$, $l_1 \geq l_2 > 0$. For each $a \in \{ 1,2 \}$, $c \in [0,l_1-l_2]$, and $b \in [0,l_1-c-1]$, set
$$
T_{a,b,c} := \ytableausetup{centertableaux}
\begin{ytableau}
a & 2 & 2 & \cdots & 2 & 3 & 3 & \cdots & 3 & 4 & 4 & \cdots & 4 \\
4 & 4 & \cdots & 4
\end{ytableau} \in \SST_4(l_1,l_2),
$$
where $b$ and $c$ are the numbers of $3$'s and $4$'s in the first row, respectively. Also, for each $c \in [0,l_1-l_2]$, set
$$
T_c := \ytableausetup{centertableaux}
\begin{ytableau}
3 & 3 & \cdots & \cdots & 3 & 4 & 4 & \cdots & 4 \\
4 & 4 & \cdots & 4
\end{ytableau} \in \SST_4(l),
$$
where $c$ is the number of $4$'s in the first row. Then, we have
\begin{align}
\begin{split}
\SST_4^{\AI}(l_1,l_2) = &\{ T_{a,b,c}, K_1(T_{a,b,c}) \mid a \in \{1,2\},\ c \in [0,l_1-l_2],\ b \in [0,l_1-c-1] \} \\
&\sqcup \{ T_c, K_1(T_c) \mid c \in [0,l_1-l_2] \},
\end{split} \nonumber
\end{align}
and hence,
\begin{align}\label{size of SST4AI(l1,l2)}
|\SST_4^{\AI}(l_1,l_2)| = 2(\sum_{c=0}^{l_1-l_2} 2(l_1-c) + (l_1-l_2+1)) = 2(l_1-l_2+1)(l_1+l_2+1).
\end{align}
Note that we have
\begin{align}
\begin{split}
K_1(T_{a,b,c}), K_1(T_c) \notin &\{ T_{a',b',c'} \mid a' \in \{1,2\},\ c' \in [0,l_1-l_2],\ b' \in [0,l_1-c'-1] \} \\
&\sqcup \{ T_{c'} \mid c' \in [0,l_1-l_2] \}
\end{split} \nonumber
\end{align}
because the $(2,1)$-th entries of $K_1(T_{a,b,c})$ and $K_1(T_c)$ are not $4$, while those of $T_{a',b',c'}$ and $T_{c'}$ are $4$.

\begin{lem}\label{Connectedness of SST4AI(l1,l2)}
Let $n = 4$ and $l_1 \geq l_2 > 0$. Then, the $\AI_3$-crystal $\SST_4^{\AI}(l_1,l_2)$ is connected.
\end{lem}

\begin{proof}
By definitions, we obtain
\begin{align}
\begin{split}
&\Btil_1 T_{a,b,c} = T_{a',b,c}, \\
&\Btil_2 T_{a,b,c} = \begin{cases}
0 \qu & \IF l_1-\delta_{a,1}-b-c \in \Z_{\ev} \AND b = 0, \\
T_{a,b-1,c} \qu & \IF l_1-\delta_{a,1}-b-c \in \Z_{\ev} \AND b > 0, \\
T_{a,b+1,c} \qu & \IF l_1-\delta_{a,1}-b-c \in \Z_{\odd} \AND b < l_1-c-1, \\
T_c \qu & \IF l_1-\delta_{a,1}-b-c \in \Z_{\odd} \AND b = l_1-c-1,
\end{cases} \\
&\Btil_3 T_{a,b,c} = \begin{cases}
K_1(T_{a',b,c}) \qu & \IF b < l_2, \\
0 \qu & \IF b - l_2 \in \Z_{\geq 0, \ev} \AND c = 0, \\
T_{a,b+1,c-1} \qu & \IF b-l_2 \in \Z_{\geq 0, \ev} \AND c > 0, \\
T_{a,b-1,c+1} \qu & \IF b-l_2 \in \Z_{\geq 0, \odd}, 
\end{cases} \\
&\Btil_1 T_c = 0, \\
&\Btil_2 T_c = T_{2,l_1-c-1,c}, \\
&\Btil_3 T_c = \begin{cases}
0 \qu & \IF l_1-l_2-c \in \Z_{\ev} \AND c = 0, \\
T_{c-1} \qu & \IF l_1-l_2-c \in \Z_{\ev} \AND c > 0, \\
T_{c+1} \qu & \IF l_1-l_2-c \in \Z_{\odd}.
\end{cases}
\end{split} \nonumber
\end{align}
where $a' \in \{ 1,2 \} \setminus \{a\}$. We show that each $T_{a,b,c}$ and $T_c$ are connected to $T_{1,0,0}$ by induction on $c$. If this is the case, then the assertion follows from the fact that $K_1$ commutes with $\Btil_i$'s and that $K_1(T_{1,0,0}) = \Btil_1 \Btil_3 T_{1,0,0}$ as verified from computation above. When $c = 0$, our claim follows from Lemma \ref{Connectedness of SST3AI(l)}. Assume that $c > 0$. By Lemma \ref{Connectedness of SST3AI(l)}, we see that $T_{a,b,c}$ is connected to $T_c$. Then, we have $\Btil_3 T_c = T_{c-1}$ (resp., $\Btil_3 \Btil_2 T_c = T_{2,l_1-c,c-1}$) if $l_1-l_2-c$ is even (resp., if $l_1-l_2-c$ is odd). Hence, our induction hypothesis implies that $T_c$ is connected to $T_{1,0,0}$. Thus, the proof completes.
\end{proof}

\section{Representation theoretic interpretation}\label{Section Representation theoretic interpretation}
In this section, we show that the characters of $\SST_n^{\AI}(\rho)$, $\rho \in \Par_m$ coincide with the characters of certain $\mathfrak{so}_n$-modules. To do so, we need results from the representation theory of the $\imath$quantum group of type $\AI$ obtained in \cite{W21b}.

\subsection{Representation theoretic interpretation of $\frgl$-crystals}
Let us briefly review the finite-dimensional representation theory of the general linear algebra $\frgl_n = \frgl_n(\C)$. The $\mathfrak{gl}_n$ is realized as the Lie algebra of $n \times n$ matrices. For each $1 \leq i,j \leq n$, let $E_{i,j}$ denote the matrix unit with entry $1$ at $(i,j)$-position. Let $e_i := E_{i,i+1}$, $f_i := E_{i+1, i}$, $i \in \{ 1,\dots,n-1 \}$ denote the Chevalley generators.

Let $M$ be a finite-dimensional $\frgl_n$-module. Then, it admits a weight space decomposition
$$
M = \bigoplus_{\lm = (\lambda_1,\dots,\lambda_n) \in X_{\frgl_n}} M_\lm, \quad M_\lambda := \{ m \in M \mid E_{i,i} m = \lambda_i m \text{ for all } i \in \{ 1,\dots,n \} \},
$$
where $X_{\frgl_n} := \Z^n$ denotes the weight lattice for $\frgl_n$. The character $\ch_{\frgl_n} M$ of $M$ is a Laurent polynomial defined by
$$
\ch_{\frgl_n} M := \sum_{\lm \in X_{\frgl_n}} (\dim M_\lm) \bfx^\lm \in \Z[x_1^{\pm 1},\ldots,x_n^{\pm 1}],
$$
where
$$
\bfx^{(\lm_1,\ldots,\lm_n)} := x_1^{\lm_1} \cdots x_n^{\lm_n}.
$$

The finite-dimensional $\frgl_n$-modules are completely reducible, and the isomorphism classes of finite-dimensional irreducible $\frgl_n$-modules are parametrized by the set
$$
X_{\frgl_n}^+ := \{ \lm = (\lm_1,\ldots,\lm_n) \in X_{\frgl_n} \mid \lm_1 \geq \cdots \geq \lm_n \}
$$
of dominant integral weights. For $\lm \in X_{\frgl_n}^+$, let $V^{\frgl_n}(\lm)$ denote the corresponding $\mathfrak{gl}_n$-module, i.e., the finite-dimensional irreducible $\mathfrak{gl}_n$-module of highest weight $\lm$.

The assignment
$$
\Par_n \rightarrow X_{\frgl_n, \geq 0}^+ := \{ \lm = (\lm_1,\ldots,\lm_n) \in X_{\frgl_n}^+ \mid \lm_n \geq 0 \};\ \lm \mapsto (\lm_1,\ldots,\lm_{\ell(\lm)},0,\ldots,0)
$$
is bijective. In this way, we often identify these two sets. In particular, for a partition $\lm \in \Par_n$, we understand $\lm_i = 0$ for $i > \ell(\lm)$. As is well-known, the $\frgl$-crystal $\SST_n(\lm)$ models $V^{\mathfrak{gl}_n}(\lm)$ in the following sense:
$$
\ch_{\frgl_n} V^{\mathfrak{gl}_n}(\lm) = \ch_{\frgl_n} \SST_n(\lm).
$$

\subsection{Representation theory of the $\imath$quantum group of type A\!I}\label{subsection results for iquantum group}
Let $\frso_n = \frso_n(\C)$ denote the special orthogonal algebra. It is realized as the Lie subalgebra of $\frgl_n$ generated by $b_i := f_i + e_i$, $i \in \{ 1,\dots,n-1 \}$ (cf. \cite[Section 4.1]{W21b}).
Recall the integer $m$ from equation \eqref{eq: def of m}.
Note that it is the rank of $\mathfrak{so}_n$.

Let $M$ be a finite-dimensional $\frso_n$-module. Then, it admits a weight space decomposition
$$
M = \bigoplus_{\nu = (\nu_1,\nu_3,\dots,\nu_{2m-1}) \in X_{\frso_n}} M_\nu, \quad M_\nu := \{ m \in M \mid b_{2i-1} m = \nu_{2i-1} m \text{ for all } i \in \{ 1,\dots,m \} \},
$$
where $X_{\frso_n} := \Z^m \sqcup (\hf + \Z)^m$ denotes the weight lattice for $\frso_n$. Note that $b_1,b_3,\dots,b_{2m-1}$ generate a Cartan subalgebra of $\mathfrak{so}_n$, and the sets $\{ b_1-b_3, b_3-b_5, \dots, b_{2m-3}-b_{2m-1}, 2b_{2m-1} \}$ and $\{ b_1-b_3,b_3-b_5,\dots,b_{2m-3}-b_{2m-1}, b_{2m-3}+b_{2m-1} \}$ form the simple coroots when $n$ is odd and when $n$ is even, respectively (cf. \cite[Section 4.1]{W21b}). A weight $\nu \in X_{\frso_n}$ is said to be an integer weight if $\nu \in \Z^m$. Let $X_{\frso_n,\Int}$ denote the set of integer weights. The character $\ch_{\frso_n} M$ of $M$ is a Laurent polynomial defined by
$$
\ch_{\frso_n} M := \sum_{\nu \in X_{\frso_n}} (\dim M_\lm) \bfy^\nu \in \Z[y_1^{\pm \hf},y_3^{\pm \hf}, \ldots,y_{2m-1}^{\pm \hf}],
$$
where
$$
\bfy^{(\nu_1,\nu_3,\ldots,\nu_{2m-1})} := y_1^{\nu_1} y_3^{\nu_3} \cdots y_{2m-1}^{\nu_{2m-1}}.
$$

The finite-dimensional $\frso_n$-modules are completely reducible, and the isomorphism classes of irreducible finite-dimensional $\frso_n$-modules are parametrized by the set
$$
X_{\frso_n}^+ := \begin{cases}
\{ (\nu_1,\nu_3,\ldots,\nu_{2m-1}) \in X_{\frso_n} \mid \nu_1 \geq \cdots \geq \nu_{2m-3} \geq |\nu_{2m-1}| \} \qu & \IF n \in \Z_{\ev}, \\
\{ (\nu_1,\nu_3,\ldots,\nu_{2m-1}) \in X_{\frso_n} \mid \nu_1 \geq \cdots \geq \nu_{2m-1} \geq 0 \} \qu & \IF n \in \Z_{\odd}
\end{cases}
$$
of dominant integral weights. For $\nu \in X_{\frso_n}^+$, let $V^{\frso_n}(\nu)$ denote the corresponding $\frso_n$-module, that is, the irreducible module of highest weight $\nu$.
Set $X_{\mathfrak{so}_n, \text{int}}^+ := X_{\mathfrak{so}_n}^+ \cap X_{\mathfrak{so}_n, \text{int}}$.

For each $\rho \in \Par_m$, set
$$
V^{\mathfrak{so}_n}(\rho) := \begin{cases}
  V^{\mathfrak{so}_n}(\nu_\rho) & \text{ if } \ell(\rho) < \frac{n}{2}, \\
  V^{\mathfrak{so}_n}(\nu_\rho^+) \oplus V^{\mathfrak{so}_n}(\nu_\rho^-) & \text{ if } \ell(\rho) = \frac{n}{2},
\end{cases}
$$
where
$$
\nu_\rho := (\rho_1,\rho_2,\dots,\rho_{\ell(\rho)},0,\dots,0), \ \nu_\rho^\pm := (\rho_1,\rho_2,\dots,\rho_{m-1}, \pm \rho_{m}) \in X_{\mathfrak{so}_n, \text{int}}^+.
$$
The rest of this section is devoted to showing that $\SST_n^{\AI}(\rho)$ models $V^{\mathfrak{so}_n}(\rho)$.

Let $\lm \in \Par_{n}$ and consider $\SST_n(\lm)$. Set
$$
\ol{\clL}(\lm) := \C \SST_n(\lm),
$$
and extend the operators $\Etil_i,\Ftil_i,\Btil_i$ on $\SST_n(\lm)$ to linear operators on $\ol{\clL}(\lm)$. As we have seen in the end of Subsection \ref{Subsection AI-crystal}, the space $\ol{\clL}(\lm)$ admits an $\frso_n$-weight space decomposition
$$
\ol{\clL}(\lm) = \bigoplus_{\nu \in X_{\frso_n,\Int}} \ol{\clL}(\lm)_\nu.
$$
This decomposition corresponds to the weight space decomposition of $V^{\mathfrak{gl}_n}(\lm)$ regarded as an $\frso_n$-module: For each $\nu \in X_{\frso_n,\Int}$, it holds that
$$
\dim V^{\mathfrak{gl}_n}(\lm)_\nu = \dim \ol{\clL}(\lm)_\nu.
$$
Hence, we have
$$
\ch_{\frso_n} V^{\mathfrak{gl}_n}(\lm) = \ch_{\AI} \SST_n(\lm).
$$
From this, we can say that the $\AI$-crystal $\SST_n(\lm)$ models the $\frso_n$-module $V^{\mathfrak{gl}_n}(\lm)$.

In \cite{W21b}, certain linear operators on $\ol{\clL}(\lm)$ are defined via the representation theory of the $\imath$quantum group of type AI. They are denoted by $\Xtil_j,\Ytil_j$ with $j \in \Itil$ (in \cite{W21b}, $\Itil$ is denoted by $\Itil_{\frk}$), where
$$
\Itil := \{ (i,+),\ (i,-) \mid i \in [1,n-2]_{\ev} \} \sqcup [1,n-1]_{\ev}.
$$
These operators are defined locally in the following sense: $\Xtil_2,\Ytil_2$ (resp., $\Xtil_{2,\pm},\Ytil_{2,\pm}$) are defined in terms of $\Etil_i,\Ftil_i$, $i \in \{1,2\}$ (resp., $i \in \{1,2,3\}$). And $\Xtil_i,\Ytil_i$, $i \in [1,n-1]_{\ev}$ (resp., $\Xtil_{i,\pm},\Ytil_{i,\pm}$, $i \in [1,n-2]_{\ev}$) are defined in the same way as $\Xtil_2,\Ytil_2$ (resp., $\Xtil_{2,\pm},\Ytil_{2,\pm}$) with the role of $\{1,2\}$ (resp., $\{1,2,3\}$)  replaced by $\{i-1,i\}$ (resp., $\{i-1,i,i+1\}$).

A nonzero weight vector $v \in \ol{\clL}(\lm)_\nu$ is said to be a highest weight vector of weight $\nu \in X_{\frso_n,\Int}^+$ if $\Xtil_j v = 0$ for all $j \in \Itil$. For a highest weight vector $v$ of weight $\nu$, set
$$
\Ytil v := \C \{ \Ytil_{j_1} \cdots \Ytil_{j_r} v \mid r \geq 0,\ j_1,\ldots,j_r \in \Itil \}.
$$
This space corresponds to an irreducible $\frso_n$-submodule of $V^{\mathfrak{gl}_n}(\lm)$ isomorphic to $V^{\frso_n}(\nu)$. Similarly, for a vector $u \in \ol{\clL}(\lm)$, set
$$
\Btil u := \C \{ \Btil_{i_1} \cdots \Btil_{i_r} u \mid r \geq 0,\ i_1,\ldots,i_r \in [1,n-1] \}.
$$
By definitions of $\Ytil_j$'s and $\Btil_i$'s, for each highest weight vector $v$, we see that
$$
\Btil v \subset \Ytil v.
$$
Also, $\Ytil v$ admits an $\frso_n$-weight space decomposition
$$
\Ytil v = \bigoplus_{\xi \in X_{\frso_n,\Int}} (\Ytil v \cap \ol{\clL}(\lm)_\xi),
$$
and it holds that
$$
\dim (\Ytil v \cap \ol{\clL}(\lm)_\xi) =  \dim V^{\frso_n}(\nu)_\xi.
$$

For each $\nu \in X_{\frso_n,\Int}^+$, choose a basis $\{ v^\nu_1,\ldots,v^\nu_{m_{\nu}} \}$ of the subspace $H(\nu) \subset \ol{\clL}(\lm)$ consisting of highest weight vectors of weight $\nu$. Then, we have
\begin{align}
\ol{\clL}(\lm) = \bigoplus_{\nu \in X_{\frso_n,\Int}^+} \bigoplus_{k=1}^{m_\nu} \Ytil v^\nu_k. \nonumber
\end{align}
By observation above, this corresponds to an irreducible decomposition of $V^{\mathfrak{gl}_n}(\lm)$ as an $\frso_n$-module.

A basis of the subspace $H(\nu)$ can be found as follows (see \cite{W21b} for detail).

\begin{defi}\label{definition of singular element}\normalfont
Let $\clB$ be an $\AI$-crystal. An element $b \in \clB$ is said to be a singular element of degree $\rho \in \Par_m$ if it satisfies the following:
\begin{enumerate}
\item\label{definition of singular element 1} $\deg_{2i-1}(b) = \rho_i$ for all $i \in [1,m]$.
\item\label{definition of singular element 2} $\deg_{2i}(b) = 0$ for all $i \in [1,m]$ such that $2i < n$.
\item\label{definition of singular element 3} $\deg_{2i+1}((\Btil_{2i-1}\Btil_{2i})^{\rho_{i+1}} b) = 0$ for all $i \in [1,m]$ such that $2i+1 < n$.
\end{enumerate}
Let $\Sing(\clB,\rho)$ denote the set of singular elements of degree $\rho$.
\end{defi}

Let $\rho \in \Par_m$, and set
$$
\Sing(\lm,\rho) := \Sing(\SST_n(\lm),\rho).
$$
For each $S \in \Sing(\lm,\rho)$, set
$$
h(S) := (1+\Btil_1)(1+\Btil_3) \cdots (1+\Btil_{2m-3}) S.
$$
Then, when $\ell(\rho) < \frac{n}{2}$ (resp., $\ell(\rho) = \frac{n}{2}$), the vector $h(S)$ (resp., $h(S)_\pm := (1 \pm \Btil_{2m-1}) h(S)$) is a highest weight vector of weight
$$
\nu := (\rho_1,\rho_2,\ldots,\rho_{\ell(\rho)},0,\ldots,0)
$$
$$
\text{(resp., $\nu_\pm := (\rho_1,\rho_2,\ldots,\rho_{m-1},\pm \rho_{m})$).}
$$
Furthermore, $\{ h(S) \mid S \in \Sing(\lm,\rho) \}$ (resp., $\{ h(S)_+ \mid S \in \Sing(\lm,\rho) \}$) forms a basis of $H(\nu)$ (resp., $H(\nu_+)$). For each $S,S' \in \Sing(\lm,\rho)$, we have $h(S) = h(S')$ if and only if $S = S'$ (resp., $h(S)_+ = h(S')_+$ if and only if $S' = \Btil_1 \Btil_3 \cdots \Btil_{2m-1} S)$. When $\ell(\rho) = \frac{n}{2}$, we have $h(S)_- = -h(S')_-$ if and only if $S' = \Btil_1 \Btil_3 \cdots \Btil_{2m-1} S$.

Now, it is convenient to set $\Sing'(\lm,\rho)$ to be $\Sing(\lm,\rho)$ if $\ell(\rho) < \frac{n}{2}$, and to be a complete set of representatives for $\Sing(\lm,\rho)/{\sim}$ with respect to the equivalence relation given by
$$
S \sim S' \text{ if and only if } S' = \Btil_1 \Btil_3 \cdots \Btil_{2m-1} S
$$
if $\ell(\rho) = \frac{n}{2}$. Then, from discussion above, we obtain
\begin{align}\label{Irreducible decomposition of V(lm) at q=infty}
\ol{\clL}(\lm) = \bigoplus_{\rho \in \Par_m} \bigoplus_{S \in \Sing'(\lm,\rho)} \Ytil h(S).
\end{align}
Furthermore, for each $S \in \Sing'(\lm,\rho)$, the subspace $\Ytil h(S)$ corresponds to an $\frso_n$-submodule of $V^{\mathfrak{gl}_n}(\lm)$ isomorphic to $V^{\mathfrak{so}_n}(\rho)$.

\begin{lem}\label{BS = Bh(S)}
Let $S \in \Sing(\lm,\rho)$. Let $m'$ denote the maximal integer such that $2m' < n$ and  $\rho_{m'} \neq 0$. Then, we have
$$
\Btil_1 \Btil_2^2 \Btil_3 \Btil_4^2 \cdots \Btil_{2m'-1} \Btil_{2m'}^2 h(S) = S.
$$
Consequently, we have
$$
\Btil S = \Btil h(S).
$$
\end{lem}

\begin{proof}
By the definition of $m'$, we have
$$
h(S) = (1 + \Btil_1)(1 + \Btil_3) \cdots (1 + \Btil_{2m'-1})S,
$$
and hence,
$$
\Btil_1 \Btil_2^2 \Btil_3 \Btil_4^2 \cdots \Btil_{2m'-1} \Btil_{2m'}^2 h(S) = \Btil_1 \Btil_2^2(1+\Btil_1) \Btil_3 \Btil_4^2(1+\Btil_3) \cdots \Btil_{2m'-1} \Btil_{2m'}^2(1+\Btil_{2m'-1})S.
$$
Here, we used Proposition \ref{degree 0 in Stembridge crystals} \eqref{degree 0 in Stembridge crystals 2}. Therefore, it suffices to show that
$$
\Btil_{2i-1} \Btil_{2i}^2(1+\Btil_{2i-1})S = S
$$
for all $i \in [1,m']$.

Let $i \in [1,m']$. Then, we have $\deg_{2i}(S) = 0$ and $\deg_{2i-1}(S) \neq 0$. By Proposition \ref{degree 0 in Stembridge crystals} \eqref{degree 0 in Stembridge crystals 1}, we have $\Btil_{2i}S = 0$ and $\Btil_{2i-1} S \neq 0$, and hence,
$$
\Btil_{2i-1} \Btil_{2i}^2(1+\Btil_{2i-1})S = \Btil_{2i-1} \Btil_{2i}^2 \Btil_{2i-1}S.
$$
By Definition \ref{definition of AI-crystals} \eqref{definition of AI-crystals 2}, we must have
$$
\deg_{2i}(\Btil_{2i-1} S) = 1,
$$
and hence,
$$
\Btil_{2i}^2 \Btil_{2i-1} S = \Btil_{2i-1} S.
$$
Therefore, we have
$$
\Btil_{2i-1} \Btil_{2i}^2 \Btil_{2i-1}S = \Btil_{2i-1}^2 S = S,
$$
as desired. This completes the proof.
\end{proof}

\subsection{Connectedness and the character of $\SST_n^{\AI}(\rho)$}
In this subsection, we show that the $\AI$-crystal $\SST_n^{\AI}(\rho)$ is connected, and its character coincides with $\ch V^{\mathfrak{so}_n}(\rho)$ for each $\rho \in \Par_m$.

Let $\rho \in \Par_m$. Define $T_\rho \in \SST_n^{\AI}(\rho)$ by
$$
T_\rho(i,j) = \begin{cases}
a_{2i-1} \qu & \IF j = 1, \\
2i \qu & \IF j > 1,
\end{cases}
$$
where $a_{2i-1} \in \{ 2i-1,2i \}$ is such that
$$
\rho_{i} - \delta_{a_{2i-1},2i-1} - \delta_{a_{2i+1},2i+1} \in \Z_{\ev}.
$$
Here, we set $a_{2\ell(\rho)+1} = 2\ell(\rho)+2$. Note that such $a_{2i-1}$'s are uniquely determined.

\begin{ex}\normalfont
Let $\rho \in \Par_m$.
\begin{enumerate}
\item If $\ell(\rho) = 1$, then
$$
T_\rho = \ytableausetup{boxsize=normal,centertableaux}
\begin{ytableau}
a & 2 & 2 & \cdots & 2
\end{ytableau}, \qu a := \begin{cases}
2 \qu & \IF \rho_1 \in \Z_{\ev}, \\
1 \qu & \IF \rho_1 \in \Z_{\odd}.
\end{cases}
$$
\item If $\ell(\rho) = 2$, then
\begin{align}
\begin{split}
&T_\rho = \ytableausetup{boxsize=normal,centertableaux}
\begin{ytableau}
a & 2 & 2 & \cdots & 2 & 2 & \cdots & 2 \\
b & 4 & 4 & \cdots & 4
\end{ytableau}, \\
&b := \begin{cases}
4 \qu & \IF \rho_2 \in \Z_{\ev}, \\
3 \qu & \IF \rho_2 \in \Z_{\odd},
\end{cases} \qu a := \begin{cases}
2 \qu & \IF \rho_1-\delta_{b,3} \in \Z_{\odd}, \\
1 \qu & \IF \rho_1-\delta_{b,3} \in \Z_{\ev}.
\end{cases}
\end{split} \nonumber
\end{align}
\end{enumerate}
\end{ex}

\begin{lem}\label{K1(Trho)}
Let $\rho \in \Par_m$ be such that $\ell(\rho) = \frac{n}{2}$. Then, we have
$$
K_1(T_\rho) = \Btil_1 \Btil_3 \cdots \Btil_{2m-1} T_\rho.
$$
\end{lem}

\begin{proof}
By Remark \ref{AI-condition in other words}, $K_1(T_\rho)$ is obtained from $T_\rho$ by applying $K$ to the first column. Since the $(i,1)$-th entry of $T_\rho$ is either $2i-1$ or $2i$ for each $i \in [1,m]$, we have
$$
K_1(T_\rho)(i,j) = \begin{cases}
a'_{2i-1} \qu & \IF j = 1, \\
2i \qu & \IF j > 1,
\end{cases}
$$
where $a'_{2i-1} \in \{ 2i-1,2i \} \setminus \{a_{2i-1}\}$. Now, the assertion is easily verified.
\end{proof}

\begin{lem}\label{singular element of degree rho in SSTn(rho)}
Let $\rho,\sigma \in \Par_m$ and $S \in \SST_n^{\AI}(\sigma) \cap \Sing(\sigma,\rho)$. Then, we have $\sigma = \rho$, and either $S = T_\rho$ or $S = K_1(T_\rho)$.
\end{lem}

\begin{proof}
Let us show by induction on $l \in [1,m]$ that in $S$, the letters $2l-1$ and $2l$ can appear only in the $l$-th row. Let $l \in [1,m]$. By our induction hypothesis, it holds that $\sh(S|_{[1,2(l-1)]}) = (\sigma_1,\sigma_2,\ldots,\sigma_{l-1})$. In particular, the $l$-th row or below consists of letters in $[2l-1,n]$.

First, we show that the $l$-th row consists of only $2l-1$ and $2l$. Assume contrary that $k > 2l$ appears in the $l$-th row. If $k$ is odd, then it must hold that $\vep_{k-1}(S) > 0$, and hence $\deg_{k-1}(S) > 0$, which contradicts Definition \ref{definition of singular element} \eqref{definition of singular element 2}. On the other hand, if $k$ is even, then it must hold that $\vep_{k-1}((\Btil_{k-2}\Btil_{k-3})^{\deg_{k-1}(S)}S) > 0$ (note that $\Btil_{k-2}$ and $\Btil_{k-3}$ do not change the entries other than $k-3,k-2,k-1$), and hence, $\deg_{k-1}((\Btil_{k-2}\Btil_{k-3})^{\deg_{k-1}(S)}S) > 0$, which contradicts Definition \ref{definition of singular element} \eqref{definition of singular element 3}. Thus, we see that the $l$-th row of $S$ consists of only $2l-1$ and $2l$.

Next, we show that
$$
S|_{\{2l-1,2l\}} = \ytableausetup{boxsize=normal,centertableaux}
\begin{ytableau}
s_l & 2l & 2l & \cdots & 2l
\end{ytableau},
$$
where $s_l := S(l,1) \in \{2l-1,2l\}$. By the argument above, we have
$$
S|_{\{2l-1,2l\}} = \ytableausetup{boxsize=normal,centertableaux}
\begin{ytableau}
\scriptstyle 2l-1 & \scriptstyle 2l-1 & \cdots & \cdots & \scriptstyle 2l-1 & 2l & 2l & \cdots & 2l \\
2l & 2l & \cdots & 2l,
\end{ytableau}
$$
where the first row consists of $\sigma_l$ boxes. By our induction hypothesis, for each $l' < l$, we have $s_{l'} := S(l',1) \in \{ 2l'-1,2l' \}$. Hence, the $(l',1)$-th entry of $K_1(S)$ is the unique letter $s'_{l'}$ in $\{ 2l'-1,2l' \} \setminus \{ s_{l'} \}$. Now, suppose that $s_{l+1} := S(l+1,1) \neq 2l$. Then, by observation above, the $(l,1)$-th entry $s'_l$ of $K_1(S)$ is the unique letter in $\{ 2l-1,2l \} \setminus \{ s_l \}$. Since $S$ is an $\AI$-tableau, we have
$$
s'_l \leq S(l,2) \leq \cdots \leq S(l,\rho_l).
$$
This, together with the semistadardness condition on the $l$-th row
$$
s_l \leq S(l,2) \leq \cdots \leq S(l,\rho_l),
$$
implies that $S(l,2) \geq 2l$. Therefore, we must have
$$
S|_{\{2l-1,2l\}} = \ytableausetup{boxsize=normal,centertableaux}
\begin{ytableau}
s_l & 2l & 2l & \cdots & 2l
\end{ytableau},
$$
as desired.

It remains to show that $s_{l+1} \neq 2l$. If $s_{l+1} = 2l$, then it must hold that $s_{l} = 2l-1$, and consequently,
$$
s'_{l} > 2l.
$$
Since $S$ is an $\AI$-tableau and the $l$-th row consists of $2l-1$ and $2l$, this implies that $\sigma_l = 1$. Therefore, the $l$-th row and below of $S$ is of the form
\begin{align}\label{first column of S}
\ytableausetup{boxsize=0.8cm,centertableaux}
\begin{ytableau}
\scriptstyle 2l-1 \\
2l \\
s_{l+2} \\
\vdots \\
s_{\ell(\sigma)}
\end{ytableau}.
\end{align}
This shows that $\deg_{2l-1}(S) = 0$. Since $\rho_l = \deg_{2l-1}(S)$ and $\rho$ is a partition, we obtain $\rho = (\rho_1,\rho_2,\ldots,\rho_{l-1})$. This implies that $\deg_i(S) = 0$ for all $i \in [2l-1,n-1]$. From \eqref{first column of S}, we see that $\deg_{2l}(S) = 0$ if and only if $s_{l+2} = 2l+1$. Proceeding in this way, we must have $s_{l+k} = 2l-1+k$ for all $k \in [1,n-2l+1]$, which is impossible because $l+(n-2l+1) = n-l+1 > m \geq \ell(\sigma)$. Thus, our claim follows.

So far, we have obtained that
$$
S(i,j) = \begin{cases}
s_i \qu & \IF j = 1, \\
2i \qu & \IF j > 1
\end{cases}
$$
for some $s_i \in \{ 2i-1,2i \}$. From this, one can easily see that
$$
\deg_{2i-1}(S) = \sigma_i
$$
for all $i \in [1,m]$. On the other hand, since $\deg_{2i-1}(S) = \rho_i$, we obtain
$$
\sigma = \rho.
$$

In order to complete the proof, we need to determine $s_l$ for all $l \in [1,\ell(\rho)]$. First, suppose that $\ell(\rho) < \frac{n}{2}$. In this case, $2\ell(\rho) \in [1,n-1]$, and hence, we must have
$$
\deg_{2\ell(\rho)}(S) = 0.
$$
This is equivalent to that
$$
\rho_{\ell(\rho)} - \delta_{s_{\ell(\rho)},2\ell(\rho)-1} \in \Z_{\ev}.
$$
Similarly, the constraint that $\deg_{2(\ell(\rho)-1)}(S) = 0$ is equivalent to that
$$
\rho_{\ell(\rho)-1} - \delta_{s_{\ell(\rho)-1},2\ell(\rho)-3} - \delta_{s_{\ell(\rho)},2\ell(\rho)-1} \in \Z_{\ev}.
$$
Proceeding in this way, we see that
$$
\rho_l - \delta_{s_l,2l-1} - \delta_{s_{l+1},2l+1} \in \Z_{\ev},
$$
for all $l \in [1,\ell(\rho)]$, where we set $s_{\ell(\rho)+1} := 2\ell(\rho)+2$. This implies that
$$
S = T_\rho,
$$
as desired.

Next, suppose that $\ell(\rho) = \frac{n}{2}$. In a similar way to above, we see that
$$
\rho_l - \delta_{s_l,2l-1} - \delta_{s_{l+1},2l+1} \in \Z_{\ev},
$$
for all $l \in [1,\ell(\rho)-1]$. This implies that we have $S = T_\rho$ if $\rho_m - \delta_{s_m,n-1} \in \Z_{\ev}$, or $S = K_1(T_\rho)$ otherwise. Thus, the proof completes.
\end{proof}

\begin{prop}\label{characterization of singular vector in terms of std}
Let $\lm \in \Par_n$, $\rho \in \Par_m$, and $S \in \Sing(\lm,\rho)$. Then, we have either $\std(S) = T_\rho$ or $\std(S) = K_1(T_\rho)$.
\end{prop}

\begin{proof}
Let $\sigma \in \Par_m$ denote the shape of $\std(S)$. Then, we have
$$
\std(S) \in \SST_n^{\AI}(\sigma) \cap \Sing(\sigma,\rho).
$$
Then, the assertion follows from Lemma \ref{singular element of degree rho in SSTn(rho)}.
\end{proof}

\begin{lem}\label{Btil singular = Ytil highest}
Let $\lm \in \Par_n$, $\rho \in \Par_m$, and $S \in \Sing(\lm,\rho)$. Then, we have $\Btil S = \Ytil h(S)$.
\end{lem}

\begin{proof}
Since we know $\Btil S = \Btil h(S) \subset \Ytil h(S)$ (by Lemma \ref{BS = Bh(S)}) and $h(S) \in \Btil S$, it suffices to show that $\Ytil_j(\Btil S) \subset \Btil S$ for all $j \in \Itil$. Furthermore, since $\Ytil_j$'s are defined locally (see Subsection \ref{subsection results for iquantum group}), it suffices to prove for the case when $n = 3$ and $j = 2$, and when $n = 4$ and $j = (2,\pm)$.

Let $(n,j) \in \{ (3,2), (4,(2,\pm)) \}$. By Proposition \ref{characterization of singular vector in terms of std} and Lemmas \ref{Connectedness of SST3AI(l)}--\ref{Connectedness of SST4AI(l1,l2)}, we see that
$$
C^{\AI}(S) \rightarrow \SST_n^{\AI}(\rho);\ T \mapsto \std(T)
$$
is a surjective morphism of $\AI$-crystals. This implies that
$$
\dim \Btil S \geq |\SST_n^{\AI}(\rho)|
$$
since we have $\Btil S = \C C^{\AI}(S)$. On the other hand, since $\Btil S \subset \Ytil h(S)$, we have
$$
\dim \Btil S \leq \dim \Ytil h(S) = \dim V^{\mathfrak{so}_n}(\rho) = |\SST_n^{\AI}(\rho)|.
$$
The last equality follows from equations \eqref{size of SST3AI(l)}--\eqref{size of SST4AI(l1,l2)}. Therefore, we obtain $\dim \Btil S = \dim \Ytil h(S)$, and hence,
$$
\Btil S = \Ytil h(S).
$$
Since the right-hand side is closed under $\Ytil_j$, so is the left-hand side. Thus, the proof completes.
\end{proof}

\begin{theo}\label{main theorem}
Let $\rho \in \Par_m$. Then, the following hold.
\begin{enumerate}
\item\label{main theorem 1} $\SST_n^{\AI}(\rho)$ is connected.
\item $\ch_{\AI} \SST_n^{\AI}(\rho) = \ch V^{\mathfrak{so}_n}(\rho)$.
\item Suppose that $\ell(\rho) < \frac{n}{2}$. Then, we have $\ch_{\AI} \SST_n^{\AI}(\rho) = \ch_{\frso_n} V^{\frso_n}(\nu_\rho)$.
\item Suppose that $\ell(\rho) = \frac{n}{2}$. Then, $\{ T + K_1(T) \mid T \in \SST_n^{\AI}(\rho) \}$ is a connected $\AI$-crystal, and $\ch_{\AI} \SST_n^{\AI}(\rho) = \ch_{\frso_n} V^{\frso_n}(\nu_\rho^+)$.
\end{enumerate}
\end{theo}

\begin{proof}
Let us prove the first assertion. By equation \eqref{Irreducible decomposition of V(lm) at q=infty} and Lemma \ref{Btil singular = Ytil highest}, we see that
$$
\ol{\clL}(\rho) = \bigoplus_{\sigma \in \Par_m} \bigoplus_{S \in \Sing'(\rho,\sigma)} \C C^{\AI}(S).
$$
This implies that for each $T \in \SST_n(\rho)$, there exists a unique $\sigma \in \Par_m$ and $S \in \Sing'(\rho,\sigma)$ such that $T \in C^{\AI}(S)$. In particular, since $\SST_n^{\AI}(\rho)$ is closed under $\Btil_i$'s, each $T \in \SST_n^{\AI}(\rho)$ is connected to an element of $\SST_n^{\AI}(\rho) \cap \Sing'(\rho,\sigma)$ for some $\sigma \in \Par_m$. By Lemma \ref{singular element of degree rho in SSTn(rho)}, we see that
$$
|\SST_n^{\AI}(\rho) \cap \Sing'(\rho,\sigma)| \leq \delta_{\rho,\sigma}.
$$
Therefore, each $T \in \SST_n^{\AI}(\rho)$ is connected to the element of $\SST_n^{\AI}(\rho) \cap \Sing'(\rho,\rho)$. This implies that $\SST_n^{\AI}(\rho)$ is connected.

Next, Let $S \in \SST_n^{\AI}(\rho) \cap \Sing'(\rho,\rho)$. From Lemma \ref{Btil singular = Ytil highest} and the first assertion, we have
$$
\Ytil h(S) = \Btil S = \C \SST_n^{\AI}(\rho).
$$
Then, the second assertion is clear from discussion after equation \eqref{Irreducible decomposition of V(lm) at q=infty}.

The third assertion is clear from the second assertion and the definitions.

Finally, assume that $\ell(\rho) = \frac{n}{2}$. Then, that $\{ T + K_1(T) \mid T \in \SST_n^{\AI}(\rho) \}$ is a connected $\AI$-crystal follows from the facts that $K_1$ is an automorphism of $\AI$-crystal on $\SST_n^{\AI}(\rho)$, and that $\SST_n^{\AI}(\rho)$ is connected. The assertion concerning characters follows from the fact that
$$
h(S) + h(K_1(S)) = h(T_\rho) + h(\Btil_1 \Btil_3 \cdots \Btil_{2m-1} T_\rho) = 2h(T_\rho)_+
$$
is a highest weight vector of weight $\nu$, where $S \in \SST_n^{\AI}(\rho) \cap \Sing'(\rho,\rho)$ (see also Lemma \ref{K1(Trho)}). Thus, the proof completes.
\end{proof}

\begin{rem}\normalfont
During the proof of Theorem \ref{main theorem}, we obtained
$$
\SST_n^{\AI}(\rho) \cap \Sing(\rho,\sigma) = \begin{cases}
\emptyset \qu & \IF \sigma \neq \rho, \\
\{T_\rho\} \qu & \IF \sigma = \rho \AND \ell(\rho) < \frac{n}{2}, \\
\{T_\rho,K_1(T_\rho)\} \qu & \IF \sigma = \rho \AND \ell(\rho) = \frac{n}{2}.
\end{cases}
$$
\end{rem}

\section{Robinson-Schensted type correspondence}\label{Section RS correspondence}
In this section, we generalize the Robinson-Schensted correspondence to the setting of $\AI$-crystals. This tells us how various $\AI$-crystals decompose into their connected components.

\subsection{Insertion scheme}
Let $\lm \in \Par_n$. Then, the map
$$
\SST_n(\lm) \otimes \SST_n(1) \rightarrow \bigsqcup_{\substack{\mu \in \Par_n \\ \lm \triangleleft \mu}} \SST_n(\mu);\ T \otimes \Sbox{l} \mapsto (T \leftarrow l)
$$
is an isomorphism of $\frgl$-crystals. The aim of this subsection is to provide an $\AI$-crystal analogue of this isomorphism, which will play a central role when generalizing the Robinson-Schensted correspondence.

Let $\rho \in \Par_m$ and consider the $\AI$-crystal $\SST_n^{\AI}(\rho) \otimes \SST_n(1)$. In order to analyze its structure, let us recall the following fact, which is the special case of \cite[Lemma 7]{Ka90}.

\begin{lem}
Let $\nu \in X_{\frso_n,\Int}^+$. For each $k \in [1,m]$, set
$$
\eps_{2k-1} := (\overbrace{0,\ldots,0}^{k-1},1,0,\ldots,0) \in X_{\frso_n,\Int}.
$$
Then, we have
$$
V^{\frso_n}(\nu) \otimes V^{\frso_n}(\eps_1) \simeq \bigoplus_{\xi} V^{\frso_n}(\nu+\xi),
$$
where $\xi$ runs through $X_{\frso_n,\Int}$ satisfying the following:
\begin{enumerate}
\item $V^{\frso_n}(\eps_1)_\xi \neq 0$.
\item $V^{\frso_n}(\eps_1)_{\xi+(1+\nu_{2i-1}-\nu_{2i+1})(\eps_{2i-1}-\eps_{2i+1})} = 0$ for all $i \in [1,m-1]$.
\item If $n \in \Z_{\ev}$, then $V^{\frso_n}(\eps_1)_{\xi+(1+\nu_{2m-3}+\nu_{2m-1})(\eps_{2m-3}+\eps_{2m-1})} = 0$.
\item If $n \in \Z_{\odd}$, then $V^{\frso_n}(\eps_1)_{\xi+(1+2\nu_{2m-1})\eps_{2m-1}} = 0$.
\end{enumerate}
\end{lem}

In terms of $\AI$-crystals, this lemma can be rewritten as follows.

\begin{lem}\label{SSTnAI(rho) otimes SSTn(1)}
Let $\rho \in \Par_m$.
\begin{enumerate}
\item Suppose that $n \in \Z_{\ev}$ and $\rho_m \neq 1$. Then, we have
$$
\SST_n^{\AI}(\rho) \otimes \SST_n(1) \simeq \bigsqcup_{\substack{\sigma \in \Par_m \\ \rho \triangleleft \sigma \OR \sigma \triangleleft \rho}} \SST_n^{\AI}(\sigma).
$$
\item\label{SSTnAI(rho) otimes SSTn(1) 2} Suppose that $n \in \Z_{\ev}$ and $\rho_m=1$. Then, we have
$$
\SST_n^{\AI}(\rho) \otimes \SST_n(1) \simeq \bigsqcup_{\substack{\sigma \in \Par_m \setminus \Par_{m-1} \\ \rho \triangleleft \sigma \OR \sigma \triangleleft \rho}} \SST_n^{\AI}(\sigma) \sqcup \SST_n^{\AI}(\rho')^2,
$$
where $\rho' := (\rho_1,\rho_2,\ldots,\rho_{m-1})$.
\item Suppose that $n \in \Z_{\odd}$ and $\rho_m = 0$. Then, we have
$$
\SST_n^{\AI}(\rho) \otimes \SST_n(1) \simeq \bigsqcup_{\substack{\sigma \in \Par_m \\ \rho \triangleleft \sigma \OR \sigma \triangleleft \rho}} \SST_n^{\AI}(\sigma).
$$
\item Suppose that $n \in \Z_{\odd}$ and $\rho_m \neq 0$. Then, we have
$$
\SST_n^{\AI}(\rho) \otimes \SST_n(1) \simeq \bigsqcup_{\substack{\sigma \in \Par_m \\ \rho = \sigma,\ \rho \triangleleft \sigma, \OR \sigma \triangleleft \rho}} \SST_n^{\AI}(\sigma).
$$
\end{enumerate}
\end{lem}

Now, let us investigate the connected components of $\SST_n^{\AI}(\rho) \otimes \SST_n(1)$. Let $T \in \SST_n^{\AI}(\rho)$ and $l \in [1,n]$. Let $\sigma \in \Par_m$ denote the shape of $\std(T \leftarrow l)$. Then, by Proposition \ref{PAI-symbol is AI-crystal morphism} and Theorem \ref{main theorem} \eqref{main theorem 1}, the connected component containing $T \otimes \Sbox{l}$ is isomorphic to $\SST_n^{\AI}(\sigma)$. This observation, together with Lemma \ref{SSTnAI(rho) otimes SSTn(1)} implies the following: Unless $n \in \Z_{\ev}$ and $\sigma = \rho'$, we have
$$
C^{\AI}(T \otimes \Sbox{l}) = \{ T' \otimes \Sbox{l'} \mid \sh(\std(T' \leftarrow l')) = \sigma \}.
$$
Hence, let us consider the case when $n \in \Z_{\ev}$ and $\sigma = \rho'$ (this can happen only when $\rho_m = 1$). In this case, the set
$$
\{ T' \otimes \Sbox{l'} \mid \sh(\std(T' \leftarrow l')) = \rho' \}
$$
consists of exactly two connected components, and both of them are isomorphic to $\SST_n^{\AI}(\rho')$. The following lemma describes one of these connected components.

\begin{lem}\label{criterion}
Let $n \in \Z_{\ev}$ and $\rho \in \Par_m$ be such that $\rho_m = 1$. Set $\rho' := (\rho_1,\ldots,\rho_{m-1})$, and $\rho'' := (\rho_1,\ldots,\rho_m,1)$. Then, the set
$$
\{ T \otimes \Sbox{l} \in \SST_n^{\AI}(\rho) \otimes \SST_n(1) \mid \sh(\std(T \leftarrow l)) = \rho' \AND \sh(T \leftarrow l) = \rho'' \}
$$
is a connected component of $\SST_n^{\AI}(\rho) \otimes \SST_n(1)$ isomorphic to $\SST_n^{\AI}(\rho')$.
\end{lem}

\begin{proof}
By Lemma \ref{SSTnAI(rho) otimes SSTn(1)} \eqref{SSTnAI(rho) otimes SSTn(1) 2}, there are exactly two elements $T_1 \otimes \Sbox{l_1},T_2 \otimes \Sbox{l_2} \in \SST_n^{\AI}(\rho) \otimes \SST_n(1)$ such that $\std(T_i \leftarrow l_i) = T_{\rho'}$, $i = 1,2$. Set $\sigma_i := \sh(T_i \leftarrow l_i)$. Since the operators $\Btil_j$, $j \in [1,n-1]$ preserves semistandard tableaux, we have
$$
\sh(T' \leftarrow l') = \sigma_i
$$
for all $T' \otimes \Sbox{l'} \in C^{\AI}(T_i \otimes \Sbox{l_i})$.

Now, we show that exactly one of $\sigma_1,\sigma_2$ is equal to $\rho''$. To do so, it suffices to prove that there is a unique tableau $T'' \in \SST_n(\rho'')$ such that $\std(T'') = T_{\rho'}$. One can easily verify that $K(T_{\rho'})$ is such a tableau. To prove the uniqueness, let $T''$ be such a tableau. Let $C_j$, $j \in [1,\rho_1]$ denote the $j$-th column of $T''$. Since $|C_1| = m+1 > m$, we have $|K_1(T'')| \leq |T''|-2 = |\rho|-1 = |\rho'| = |\std(T'')|$. Hence, by Lemma \ref{observation of C1C2}, we have
$$
\std(T'') = K_1(T'') = P(K(C_1) \otimes C_2 \otimes \cdots \otimes C_{\rho_1}).
$$
Since $\std(T'') = T_{\rho'}$, the length of the first column of $\std(T'')$ is $m-1$. On the other hand, we have $|K(C_1)| = m-1$. Hence, it must hold that
$$
P(K(C_1) \otimes C_2 \otimes \cdots \otimes C_{\rho_1}) = K(C_1) C_2 \cdots C_{\rho_1}.
$$
Therefore, we obtain
$$
K(C_1) C_2 \cdots C_{\rho_1} = T_{\rho'}.
$$
This implies that $T'' = C_1 C_2 \cdots C_{\rho_1}$ is obtained from $T_{\rho'}$ by applying $K$ to the first column. In particular, $T''$ is uniquely determined. Thus, the proof completes.
\end{proof}

Combining Lemmas \ref{SSTnAI(rho) otimes SSTn(1)} and \ref{criterion}, we obtain the following.

\begin{prop}
Let $\rho \in \Par_m$.
\begin{enumerate}
\item Suppose that $n \in \Z_{\ev}$ and $\rho_m \neq 1$. Then, the map
$$
\SST_n^{\AI}(\rho) \otimes \SST_n(1) \rightarrow \bigsqcup_{\substack{\sigma \in \Par_m \\ \rho \triangleleft \sigma \OR \sigma \triangleleft \rho}} \SST_n^{\AI}(\sigma); \ T \otimes \Sbox{l} \mapsto \std(T \leftarrow l)
$$
is an isomorphism of $\AI$-crystals.
\item Suppose that $n \in \Z_{\ev}$ and $\rho_m=1$. Then, the map
\begin{align}
\begin{split}
&\SST_n^{\AI}(\rho) \otimes \SST_n(1) \rightarrow \bigsqcup_{\substack{\sigma \in \Par_m \setminus \Par_{m-1} \\ \rho \triangleleft \sigma \OR \sigma \triangleleft \rho}} \SST_n^{\AI}(\sigma) \sqcup (\SST_n^{\AI}(\rho') \times \{ +,- \}) \\
&T \otimes \Sbox{l} \mapsto \begin{cases}
\std(T \leftarrow l) & \IF \ell(\sh(\std(T \leftarrow l))) = m, \\
(\std(T \leftarrow l),+) & \IF \ell(\sh(\std(T \leftarrow l))) < m \AND \ell(\sh(T \leftarrow l)) > m, \\
(\std(T \leftarrow l),-) & \IF \ell(\sh(\std(T \leftarrow l))) < m \AND \ell(\sh(T \leftarrow l)) = m
\end{cases}
\end{split} \nonumber
\end{align}
is an isomorphism of $\AI$-crystals, where $\rho' := (\rho_1,\rho_2,\ldots,\rho_{m-1})$.
\item Suppose that $n \in \Z_{\odd}$ and $\rho_m = 0$. Then, the map
$$
\SST_n^{\AI}(\rho) \otimes \SST_n(1) \rightarrow \bigsqcup_{\substack{\sigma \in \Par_m \\ \rho \triangleleft \sigma \OR \sigma \triangleleft \rho}} \SST_n^{\AI}(\sigma); \ T \otimes \Sbox{l} \mapsto \std(T \leftarrow l)
$$
is an isomorphism of $\AI$-crystals.
\item Suppose that $n \in \Z_{\odd}$ and $\rho_m \neq 0$. Then, the map
$$
\SST_n^{\AI}(\rho) \otimes \SST_n(1) \rightarrow \bigsqcup_{\substack{\sigma \in \Par_m \\ \rho = \sigma, \ \rho \triangleleft \sigma, \OR \sigma \triangleleft \rho}} \SST_n^{\AI}(\sigma); \ T \otimes \Sbox{l} \mapsto \std(T \leftarrow l)
$$
is an isomorphism of $\AI$-crystals.
\end{enumerate}
\end{prop}

\subsection{Robinson-Schensted type correspondence}
Given a word $w = (w_1,\ldots,w_d) \in \clW$, define its $P^{\AI}$-symbol $P^{\AI}(w)$ by
$$
P^{\AI}(w) := \std(P(w)).
$$
For each $k \in [0,d]$, set
$$
P^{\AI,k} := P^{\AI}(w_1,\ldots,w_k)
$$
and
$$
\rho^k := \sh(P^{\AI,k}).
$$
Also, define $Q^{\AI, k}$ to be the pair $(Q^{\AI, k}_1,Q^{\AI, k}_2)$ of a standard tableau $Q^{\AI,k}_1 \in \STab_k(\rho^k)$ and a set $Q^{\AI,k}_2$ of subsets of $([1,k] \setminus \{ Q^{\AI,k}_1(i,j) \mid (i,j) \in D(\rho^k) \}) \sqcup \{+,-\}$ inductively as follows (cf. \cite[Section 8]{Sun86}). First, set $Q^{\AI,0} := (\emptyset,\emptyset)$. Next, note that we have either $\rho^{k} = \rho^{k-1}$, $\rho^{k-1} \triangleleft \rho^k$, or $\rho^k \triangleleft \rho^{k-1}$. Suppose that $\rho^k = \rho^{k-1}$. Then, we set
$$
Q^{\AI, k}_1 := Q^{\AI, k-1}_1, \qu Q^{\AI, k}_2 := Q^{\AI, k-1}_2 \sqcup \{ \{k\} \}.
$$
Next, suppose that $\rho^{k-1} \triangleleft \rho^k$. Then, we set
$$
Q^{\AI, k}_1(i,j) := \begin{cases}
Q^{\AI, k-1}_1(i,j) \qu & \IF (i,j) \in D(\rho^{k-1}), \\
k \qu & \IF (i,j) \in D(\rho^k) \setminus D(\rho^{k-1}),
\end{cases} \qu Q^{\AI, k}_2 := Q^{\AI, k-1}_2.
$$
Finally, suppose that $\rho^k \triangleleft \rho^{k-1}$. Then, we define $Q^{\AI, k}_1$ to be the unique tableau of shape $\rho^k$ such that $(Q^{\AI, k}_1 \leftarrow l) = Q^{\AI, k-1}_1$ for some $l \in [1,k-1]$ (such $Q^{\AI, k}_1$ and $l$ can be computed by the inverse of the insertion algorithm), and
$$
Q^{\AI, k}_2 := \begin{cases}
Q^{\AI, k-1}_2 \sqcup \{ \{ l,k,+ \} \} \qu & \IF \ell(\rho^{k-1}) = \frac{n}{2} > \ell(\rho^k) \AND \ell(\sh(P^{\AI,k-1} \leftarrow w_k) > m, \\
Q^{\AI, k-1}_2 \sqcup \{ \{ l,k,- \} \} \qu & \IF \ell(\rho^{k-1}) = \frac{n}{2} > \ell(\rho^k) \AND \ell(\sh(P^{\AI,k-1} \leftarrow w_k) = m, \\
Q^{\AI, k-1}_2 \sqcup \{ \{ l,k \} \} \qu & \OW.
\end{cases}
$$
Now, define the $Q^{\AI}$-symbol $Q^{\AI}(w) = (Q^{\AI}(w)_1,Q^{\AI}(w)_2)$ and the $\AI$-shape $\sh^{\AI}(w)$ of $w$ to be $(Q^{\AI, d}_1,Q^{\AI, d}_2)$ and $\rho^d$, respectively.

\begin{ex}\normalfont\label{ex: RS for AI}
\ \begin{enumerate}
\item\label{ex: RS for AI 1} Let $n = 4$ and $w = (1,1,4,2,1,1,1)$. Then, $(P^{\AI}(w),Q^{\AI}(w))$ is calculated as follows: The right-most tableaux of the first and second row are $P^{\AI}(w)$ and $Q^{\AI}(w)_1$, respectively, and the sets in the third row are the elements of $Q^{\AI}(w)_2$.
$$
\xymatrix@C=8pt@R=10pt{
\emptyset& \text{$\ytableausetup{smalltableaux}
\begin{ytableau}
1
\end{ytableau}$}& \emptyset &\text{$\ytableausetup{smalltableaux}
\begin{ytableau}
4
\end{ytableau}$}&\text{$\ytableausetup{smalltableaux}
\begin{ytableau}
2 \\
4
\end{ytableau}$}&\text{$\ytableausetup{smalltableaux}
\begin{ytableau}
3
\end{ytableau}$}&\text{$\ytableausetup{smalltableaux}
\begin{ytableau}
1 \\
3
\end{ytableau}$}&\text{$\ytableausetup{smalltableaux}
\begin{ytableau}
3
\end{ytableau}$} \\
\emptyset&\text{$\ytableausetup{smalltableaux}
\begin{ytableau}
1
\end{ytableau}$}& \emptyset &\text{$\ytableausetup{smalltableaux}
\begin{ytableau}
3
\end{ytableau}$}&\text{$\ytableausetup{smalltableaux}
\begin{ytableau}
3 \\
4
\end{ytableau}$}&\text{$\ytableausetup{smalltableaux}
\begin{ytableau}
4
\end{ytableau}$}&\text{$\ytableausetup{smalltableaux}
\begin{ytableau}
4 \\
6
\end{ytableau}$}&\text{$\ytableausetup{smalltableaux}
\begin{ytableau}
6
\end{ytableau}$} \\
&&\{1,2\}& & & \{3,5,+\}& &\{4,7,-\}
}
$$
\item Let $n = 5$ and $w = (1,1,4,2,1)$. Then, $(P^{\AI}(w),Q^{\AI}(w))$ is calculated as follows:
$$
\xymatrix@C=8pt@R=10pt{
\emptyset& \text{$\ytableausetup{smalltableaux}
\begin{ytableau}
1
\end{ytableau}$}& \emptyset &\text{$\ytableausetup{smalltableaux}
\begin{ytableau}
4
\end{ytableau}$}&\text{$\ytableausetup{smalltableaux}
\begin{ytableau}
2 \\
4
\end{ytableau}$}&\text{$\ytableausetup{smalltableaux}
\begin{ytableau}
3 \\
5
\end{ytableau}$} \\
\emptyset&\text{$\ytableausetup{smalltableaux}
\begin{ytableau}
1
\end{ytableau}$}& \emptyset &\text{$\ytableausetup{smalltableaux}
\begin{ytableau}
3
\end{ytableau}$}&\text{$\ytableausetup{smalltableaux}
\begin{ytableau}
3 \\
4
\end{ytableau}$}&\text{$\ytableausetup{smalltableaux}
\begin{ytableau}
3 \\
4
\end{ytableau}$} \\
&&\{1,2\}&&&\{5\}
}
$$
\end{enumerate}
\end{ex}

The $Q^{\AI}$-symbols for various words are characterized by the notion of $\mathfrak{so}_n$-oscillating tableaux, which we now define.

\begin{defi}\normalfont
Let $\rho \in \Par_m$ and $d \geq 0$. An $\mathfrak{so}_n$-oscillating tableau of shape $\rho$ and length $d$ is a sequence $\bfrho = ((\rho^0,s^0),(\rho^1,s^1),\ldots,(\rho^d,s^d))$ of pairs $(\rho^k,s^k) \in \Par_m \times \{ 0,+,- \}$ satisfying the following:
\begin{enumerate}
\item $\rho^0 = \emptyset$, $\rho^d = \rho$, and either $\rho^k = \rho^{k-1}$, $\rho^{k-1} \triangleleft \rho^k$, or $\rho^k \triangleleft \rho^{k-1}$.
\item If $\rho^k = \rho^{k-1}$ for some $k$, then $n \in \Z_{\odd}$ and $\ell(\rho^k) = m$.
\item If $s^k \in \{+,-\}$ for some $k$, then $n \in \Z_{\ev}$
\item $s^k \in \{+,-\}$ if and only if $\ell(\rho^{k-1}) = m$, and $\ell(\rho^k) = m-1$.
\end{enumerate}
Let $\OT_n(\rho)$ denote the set of $\mathfrak{so}_n$-oscillating tableaux of shape $\rho$, and set
$$
\OT_n := \bigsqcup_{\rho \in \Par_m} \OT_n(\rho).
$$
\end{defi}

Given an $\mathfrak{so}_n$-oscillating tableau $\bfrho = ((\rho^0,s^0),(\rho^1,s^1),\ldots,(\rho^d,s^d))$ of shape $\rho$ and length $d$, we associate a pair $Q(\bfrho) = (Q_1,Q_2)$ of a standard tableau $Q_1 \in \STab_d(\rho)$ and a set $Q_2$ of subsets of $([1,d] \setminus \{ Q_1(i,j) \mid (i,j) \in D(\rho) \}) \sqcup \{ +,- \}$ as follows. When $d = 0$, define $Q(\bfrho) = (\emptyset,\emptyset)$. Now, assume that $d > 0$ and set $\bfrho' := ((\rho^0,s^0),(\rho^1,s^1),\ldots,(\rho^{d-1},s^{d-1}))$ and $Q(\bfrho') = (Q'_1,Q'_2)$. Then, $Q(\bfrho)$ is defined as follows:
\begin{enumerate}
\item If $\rho^{d} = \rho^{d-1}$, then $Q_1 = Q'_1$ and $Q_2 = Q'_2 \sqcup \{ \{d\} \}$.
\item If $\rho^{d-1} \triangleleft \rho^d$, then
$$
Q_1(i,j) = \begin{cases}
Q'_1(i,j) \qu & \IF (i,j) \in D(\rho^{d-1}), \\
d \qu & \IF (i,j) \in D(\rho^d) \setminus D(\rho^{d-1}),
\end{cases}
$$
and $Q_2 = Q'_2$.
\item If $\rho^d \triangleleft \rho^{d-1}$ and $s^d = 0$, then $Q'_1 = (Q_1 \leftarrow l)$ for some uniquely determined $l \in [1,d-1]$, and $Q_2 = Q'_2 \sqcup \{ \{ l,d \} \}$.
\item If $\rho^d \triangleleft \rho^{d-1}$ and $s^d \in \{+,-\}$, then $Q'_1 = (Q_1 \leftarrow l)$ for some uniquely determined $l \in [1,d-1]$, and $Q_2 = Q'_2 \sqcup \{ \{ l,d,s^d \} \}$.
\end{enumerate}
As before, the tableau $Q'_1$ and letter $l$ in items $(3)$ and $(4)$ above can be computed by the inverse of the insertion algorithm.

\begin{ex}\normalfont
  Set $\rho := (1)$, and
  $$
  \bfrho := ((\emptyset,0), ((1),0), (\emptyset,0), ((1),0), ((1,1),0), ((1),+), ((1,1),0), ((1),-))
  $$
  Then, we have $\bfrho \in \OT_4(\rho)$, and $Q(\bfrho)$ coincides with $Q^{\AI}(w)$ in Example \ref{ex: RS for AI} \eqref{ex: RS for AI 1}.
\end{ex}

\begin{lem}
The assignment $\bfrho \mapsto Q(\bfrho)$ is injective.
\end{lem}

\begin{proof}
Let $\bfrho,\bfsigma \in \OT_n$ be such that $Q(\bfrho) = Q(\bfsigma)$. Set
$$
(Q_1,Q_2) := Q(\bfrho), \qu \rho := \sh(Q_1).
$$
First of all, we see that the shapes of $\bfrho$ and $\bfsigma$ coincide; both are $\rho$.

From the definition of $Q(\bfrho)$, we see that the length of $\bfrho$ is the maximal integer appearing in $Q(\bfrho)$. In other words, the length of $\bfrho$ is determined by $Q(\bfrho)$. Therefore, $\bfrho$ and $\bfsigma$ have the same length, say $d$.

We prove that $\bfrho = \bfsigma$ by induction on $d$; the case when $d = 0$ is clear. Assume that $d > 0$, and set $\bfrho' = ((\rho^0,s^0),(\rho^1,s^1),\ldots,(\rho^{d-1},s^{d-1}))$. Then, we have
$$
(Q(\bfrho'),s^d) = \begin{cases}
((Q_1,Q_2 \setminus \{\{d\}\}),0) \qu & \IF \{d\} \in Q_2, \\
((Q_1|_{[1,d-1]}, Q_2),0) \qu & \IF d \in Q_1, \\
(((Q_1 \leftarrow l), Q_2 \setminus \{ \{l,d \} \}),0) \qu & \IF \{ l,d \} \in Q_2 \Forsome l < d, \\
(((Q_1 \leftarrow l), Q_2 \setminus \{ \{l,d,\pm \} \}),\pm) \qu & \IF \{ l,d,\pm \} \in Q_2 \Forsome l < d.
\end{cases}
$$
This implies that $Q(\bfrho')$ and $s^d$ are determined by $Q(\bfrho)$. By our induction hypothesis, $\bfrho'$ is determined by $Q(\bfrho')$. Hence, $\bfrho$ is determined by $Q(\bfrho)$, which implies $\bfrho = \bfsigma$, as desired. This completes the proof.
\end{proof}

In this way, we often identify $\bfrho$ with $Q(\bfrho)$. Now, the following is immediate from the results obtained so far.

\begin{theo}
Let $n \geq 3$. Then, the assignment
$$
\RS^{\AI} : \clW \rightarrow \bigsqcup_{\rho \in \Par_m} \SST_n^{\AI}(\rho) \times \OT_n(\rho);\ w \mapsto (P^{\AI}(w),Q^{\AI}(w))
$$
is an isomorphism of $\AI$-crystals.
\end{theo}

\subsection{Branching rule}
We have obtained two isomorphisms of $\AI$-crystals:
$$
\RS : \clW \rightarrow \bigsqcup_{\lm \in \Par_n} \SST_n(\lm) \times \STab_{|\lm|}(\lm),
$$
and
$$
\RS^{\AI} : \clW \rightarrow \bigsqcup_{\rho \in \Par_m} \SST_n^{\AI}(\rho) \times \OT_n(\rho).
$$
Let $\lm \in \Par_n$, $\rho \in \Par_m$, $(P,Q) \in \SST_n(\lm) \times \STab_{|\lm|}(\lm)$, $(P',Q') \in \SST_n^{\AI}(\rho) \times \OT_n(\rho)$ be such that
$$
(P',Q') = \RS^{\AI} \circ \RS\inv(P,Q).
$$
Since $\SST_n^{\AI}(\rho)$ is connected, for each $P'' \in \SST_n^{\AI}(\rho)$, there exist $i_1,\ldots,i_r \in [1,n-1]$ such that
$$
P'' = \Btil_{i_1} \cdots \Btil_{i_r} P'.
$$
Therefore, we obtain
$$
(P'',Q') = \Btil_{i_1} \cdots \Btil_{i_r}(P',Q') = \RS^{\AI} \circ \RS\inv(\Btil_{i_1} \cdots \Btil_{i_r} P,Q).
$$
This implies that $Q$ is independent of $P'$ and uniquely determined by $Q'$. Let us say $Q$ is the $Q$-symbol of $Q'$, and write it as $Q(Q')$. Hence, there exists a map
$$
Q : \bigsqcup_{\rho \in \Par_m} \OT_n(\rho) \rightarrow \bigsqcup_{\lm \in \Par_n} \STab_{|\lm|}(\lm);\ Q' \mapsto Q(Q').
$$

\begin{theo}
Let $\lm \in \Par_n$. Then, the map
\begin{align}
\begin{split}
\SST_n(\lm) &\rightarrow \bigsqcup_{\rho \in \Par_m} \SST_n^{\AI}(\rho) \times \{ Q' \in \OT_n(\rho) \mid Q(Q') = T(\lm) \}; \\
T &\mapsto (P^{\AI}(T), Q^{\AI}(\CR(T)))
\end{split} \nonumber
\end{align}
is an isomorphism of $\AI$-crystals, where $T(\lm) \in \STab_{|\lm|}(\lm)$ is given by
$$
T(\lm)(i,j) := \sum_{k=1}^{j-1} d_k + i,
$$
and $d_k$ denotes the length of the $k$-th column of $D(\lm)$.
\end{theo}

\begin{proof}
The assignment $T \mapsto (P^{\AI}(T),Q^{\AI}(\CR(T)))$ factors
$$
\SST_n(\lm) \xrightarrow[]{\CR} \clW \xrightarrow[]{\RS^{\AI}} \bigsqcup_{\rho \in \Par_m} \SST_n^{\AI}(\rho) \times \OT_n(\rho).
$$
Hence, it suffices to show that
$$
\{ Q^{\AI}(\CR(T)) \mid T \in \SST_n(\lm) \} = \bigsqcup_{\rho \in \Par_m} \{ Q' \in \OT_n(\rho) \mid Q(Q') = T(\lm) \}
$$

Let $T \in \SST_n(\lm)$. Since $Q(\CR(T)) = T(\lm)$, we see that
$$
Q(Q^{\AI}(\CR(T))) = T(\lm)
$$
for all $T \in \SST_n(\lm)$. On the other hand, if $Q' \in \OT_n(\rho)$ is such that $Q(Q') = T(\lm)$, then there exist $w \in \clW$ and $P \in \SST_n^{\AI}(\rho)$ such that
$$
P^{\AI}(w) = P,\ Q^{\AI}(w) = Q', \AND Q(w) = T(\lm).
$$
This shows that if we set $T := P(w)$, then we have $\CR(T) = w$, and hence,
$$
Q^{\AI}(\CR(T)) = Q^{\AI}(w) = Q'.
$$
This completes the proof.
\end{proof}

\begin{cor}\label{branching rule}
Let $\lm \in X_{\frgl_n,\geq 0}^+$ and $\nu \in X_{\frso_n,\Int}^+$. Then, we have
$$
V^{\frgl_n}(\lm) \simeq \bigoplus_{\nu \in X_{\frso_n,\Int}^+} V^{\frso_n}(\nu)^{\oplus [\lm:\nu]},
$$
where
$$
[\lm:\nu] := \sharp \{ Q' \in \OT_n(\nu_1,\nu_3,\ldots,\nu_{2m-3},|\nu_{2m-1}|) \mid Q(Q') = T(\lm) \}.
$$
\end{cor}

\begin{ex}\normalfont
  Let $n = 3$, $\lambda = (2,1)$.
  Then, the $\mathfrak{so}_n$-oscillating tableaux of length $|\lambda| = 3$ are the following (we omit the second factors as they are all $0$):
  \begin{align*}
    \begin{split}
      &Q'_1 = (\emptyset, (1), (2), (3)), \\
      &Q'_2 = (\emptyset, (1), (2), (2)), \\
      &Q'_3 = (\emptyset, (1), (2), (1)), \\
      &Q'_4 = (\emptyset, (1), (1), (2)), \\
      &Q'_5 = (\emptyset, (1), (1), (1)), \\
      &Q'_6 = (\emptyset, (1), (1), \emptyset), \\
      &Q'_7 = (\emptyset, (1), \emptyset, (1)).
    \end{split}
  \end{align*}
  We have
  \begin{align*}
    \begin{split}
      &(\ytableausetup{smalltableaux}
      \begin{ytableau}
      1 & 2 & 2
      \end{ytableau}, Q'_1) \xrightarrow{(\RS^{\AI})^{-1}} (1,2,2) \xrightarrow{\RS} (\ytableausetup{smalltableaux}
      \begin{ytableau}
      1 & 2 & 2
      \end{ytableau}, \ytableausetup{smalltableaux}
      \begin{ytableau}
      1 & 2 & 3
      \end{ytableau}), \\
      &(\ytableausetup{smalltableaux}
      \begin{ytableau}
      1 & 2
      \end{ytableau}, Q'_2) \xrightarrow{(\RS^{\AI})^{-1}} (2,3,2) \xrightarrow{\RS} (\ytableausetup{smalltableaux}
      \begin{ytableau}
      2 & 2 \\
      3
      \end{ytableau}, \ytableausetup{smalltableaux}
      \begin{ytableau}
      1 & 2 \\
      3
      \end{ytableau}), \\
      &(\ytableausetup{smalltableaux}
      \begin{ytableau}
      1
      \end{ytableau}, Q'_3) \xrightarrow{(\RS^{\AI})^{-1}} (2,2,1) \xrightarrow{\RS} (\ytableausetup{smalltableaux}
      \begin{ytableau}
      1 & 2 \\
      2
      \end{ytableau}, \ytableausetup{smalltableaux}
      \begin{ytableau}
      1 & 2 \\
      3
      \end{ytableau}), \\
      &(\ytableausetup{smalltableaux}
      \begin{ytableau}
      1 & 2
      \end{ytableau}, Q'_4) \xrightarrow{(\RS^{\AI})^{-1}} (3,2,2) \xrightarrow{\RS} (\ytableausetup{smalltableaux}
      \begin{ytableau}
      2 & 2 \\
      3
      \end{ytableau}, \ytableausetup{smalltableaux}
      \begin{ytableau}
      1 & 3 \\
      2
      \end{ytableau}), \\
      &(\ytableausetup{smalltableaux}
      \begin{ytableau}
      1
      \end{ytableau}, Q'_5) \xrightarrow{(\RS^{\AI})^{-1}} (2,1,2) \xrightarrow{\RS} (\ytableausetup{smalltableaux}
      \begin{ytableau}
      1 & 2 \\
      2
      \end{ytableau}, \ytableausetup{smalltableaux}
      \begin{ytableau}
      1 & 3 \\
      2
      \end{ytableau}), \\
      &(\emptyset, Q'_6) \xrightarrow{(\RS^{\AI})^{-1}} (3,2,1) \xrightarrow{\RS} (\ytableausetup{smalltableaux}
      \begin{ytableau}
      1 \\
      2 \\
      3
      \end{ytableau}, \ytableausetup{smalltableaux}
      \begin{ytableau}
      1 \\
      2 \\
      3
      \end{ytableau}), \\
      &(\ytableausetup{smalltableaux}
      \begin{ytableau}
      1
      \end{ytableau}, Q'_7) \xrightarrow{(\RS^{\AI})^{-1}} (1,1,1) \xrightarrow{\RS} (\ytableausetup{smalltableaux}
      \begin{ytableau}
      1 & 1 & 1
      \end{ytableau}, \ytableausetup{smalltableaux}
      \begin{ytableau}
      1 & 2 & 3
      \end{ytableau}).
    \end{split}
  \end{align*}
  Therefore, we obtain
  $$
  [\lambda; (3)] = 0,\ [\lambda; (2)] = 1,\ [\lambda; (1)] = 1,\ [\lambda; \emptyset] = 0.
  $$
  Compare this result with the $\AI$-crystal graph of $\SST(\lambda)$ in equation \eqref{eq: AIcrystal graph of 2,1}.
\end{ex}

Corollary \ref{branching rule} gives a combinatorial cancellation-free branching rule from $\frgl_n$ to $\frso_n$; the multiplicity $[\lm:\mu]$ is equal to the number of $\mathfrak{so}_n$-oscillating tableaux satisfying certain conditions. Jang and Kwon \cite{JK21} proved that the multiplicity $[\lm:\nu]$ is equal to the number of Littlewood-Richardson tableaux satisfying certain conditions. Their description of the multiplicities generalizes famous Littlewood's branching rule \cite{Li44,Li50}. It would be interesting to compare their multiplicity set with ours.

Jagenteufel \cite{J20} found an algorithm which gives a bijection from $\OT_n(\rho)$ (with $n$ odd) to the set of pairs of a standard Young tableau of shape $\lm$ in alphabet $[1,|\lm|]$ and an orthogonal Littlewood-Richardson tableau associated to $\lm,\mu$ (in \cite{J20}, $\mathfrak{so}_n$-oscillating tableaux are called vacillating tableaux). This algorithm may be used to make the assignment $Q' \mapsto Q(Q')$ more explicit.


\begin{thebibliography}{99}
\bibitem{BS17} D. Bump and A. Schilling, Crystal Bases, Representations and Combinatorics, World Scientific Publishing Co. Pte. Ltd., Hackensack, NJ, 2017. xii+279 pp.

\bibitem{GK91} A. M. Gavrilik and A. U. Klimyk, $q$-deformed orthogonal and pseudo-orthogonal algebras and their representations, Lett. Math. Phys. 21 (1991), no. 3, 215--220.

\bibitem{J20} J. Jagenteufel, A Sundaram type bijection for $SO(2k+1)$: vacillating tableaux and pairs consisting of a standard Young tableau and an orthogonal Littlewood-Richardson tableau, S\'{e}m. Lothar. Combin. 82B (2020), Art. 33, 12 pp.

\bibitem{JK21} I.-S. Jang and J.-H. Kwon, Flagged Littlewood-Richardson tableaux and branching rule for classical groups, J. Combin. the Theory Ser. A 181 (2021), 105419, 51 pp.

\bibitem{Ka90} M. Kashiwara, Crystalizing the $q$-analogue of universal enveloping algebras, Comm. Math. Phys. 133 (1990), no. 2, 249--260.

\bibitem{KN94} M. Kashiwara and T. Nakashima, Crystal graphs for representations of the $q$-analogue of classical Lie algebras, J. Algebra 165 (1994), no. 2, 295--345.

\bibitem{KES83} R. C. King and N. G. I. El-Sharkaway, Standard Young tableaux and weight multiplicities of the classical Lie groups, J. Phys. A 16 (1983), no. 14, 3153--3177.

\bibitem{KT90} K. Koike and I. Terada, Young diagrammatic methods for the restriction of representations of complex classical Lie groups to reductive subgroups of maximal rank, Adv. Math. 79 (1990), no. 1, 104--135.

\bibitem{Lec03} C. Lecouvey, Schensted-type correspondences and plactic monoids for types $B_n$ and $D_n$, J. Algebraic Combin. 18 (2003), no. 2, 99--133.

\bibitem{Le99} G. Letzter, Symmetric pairs for quantized enveloping algebras, J. Algebra 220 (1999), no. 2, 729--767.

\bibitem{Li44} D. E. Littlewood, On invariant the theory under restricted groups.
Philos. Trans. Roy. Soc. London Ser. A 239 (1944), 387--417.

\bibitem{Li50} D. E. Littlewood, The the Theory of Group Characters and Matrix Representations of Groups, Reprint of the second (1950) edition. AMS Chelsea Publishing, Providence, RI, 2006. viii+314 pp.

\bibitem{O91} S. Okada, A Robinson-Schensted-type algorithm for $SO(2n,C)$, J. Algebra 143 (1991), no. 2, 334--372.

\bibitem{Pr90} R. A. Proctor, A Schensted algorithm which models tensor representations of the orthogonal group, Canad. J. Math. 42 (1990), no. 1, 28--49.

\bibitem{Sun86} S. Sundaram, On The Combinatorics of Representations of The Symplectic Group, Thesis (Ph.D.)-Massachusetts Institute of Technology. 1986.

\bibitem{Sun90} S. Sundaram, Orthogonal tableaux and an insertion algorithm for $SO(2n+1)$, J. Combin. the Theory Ser. A 53 (1990), no. 2, 239--256.

\bibitem{W20} H. Watanabe, Crystal basis the theory for a quantum symmetric pair $(\U,\U^\jmath)$, Int. Math. Res. Not. IMRN 2020, no. 22, 8292--8352.

\bibitem{W21b} H. Watanabe, Based modules over the $\imath$quantum group of type AI, to appear in Math. Z.
\end{thebibliography}
\end{document}